\documentclass[11pt]{amsart}
\usepackage[utf8]{inputenc}
\usepackage{fontenc}
\usepackage{amsfonts}
\usepackage{amssymb}
\usepackage{amsmath}
\usepackage{amsthm}\usepackage{mathtools}
\usepackage{enumerate}
\usepackage{hyperref}\usepackage{enumitem}
\usepackage{mathrsfs}
\usepackage{tikz}\usetikzlibrary{calc}
\usepackage{marginnote}
\usepackage{xcolor,enumitem}
\usepackage{soul}\usepackage{tikz}
\usepackage[all,cmtip]{xy}

\newcommand{\e}{\varepsilon}

\newcommand{\N}{\mathbb{N}}
\newcommand{\C}{\mathbb{C}}

\newcommand{\cV}{\mathcal{V}}

\newcommand{\cE}{\mathcal{E}}

\newcommand{\eps}{\varepsilon}

\newtheorem*{acknowledgments}{Acknowledgments}

\newcommand{\NN}{\mathbb{N}}

\newcommand{\cstu}{\mathrm{C}^*_u}
\newcommand{\csts}{\mathrm{C}^*_s}

\newtheorem*{rigprob*}{Rigidity Problem for uniform Roe Algebras}
\newtheorem*{rigprobcorona*}{Rigidity Problem for uniform Roe Coronas}

\newcommand{\cstar}{$\mathrm{C}^*$}

\newcommand{\bbP}{\mathbb P}
\newcommand{\cC}{\mathcal{C}}
\newcommand{\cU}{\mathcal{U}}

\newcommand{\cF}{\mathcal{F}}
\newcommand{\cP}{\mathcal{P}}
\newcommand{\bbN}{\mathbb{N}}
\newcommand{\cB}{\mathcal{B}}
\newcommand{\cK}{\mathcal{K}}

\newcommand{\cA}{\mathcal A} 

\newtheorem{theorem}{Theorem}[section]
\newtheorem*{theorem*}{Theorem}
\newtheorem{proposition}[theorem]{Proposition}
\newtheorem{problem}[theorem]{Problem}
\newtheorem*{proposition*}{Proposition}
\newtheorem{lemma}[theorem]{Lemma}
\newtheorem*{lemma*}{Lemma}
\newtheorem{corollary}[theorem]{Corollary}
\newtheorem*{corollary*}{Corollar}

\newtheorem*{fact*}{Fact}
\theoremstyle{definition}
\newtheorem{definition}[theorem]{Definition}
\newtheorem*{definition*}{Definition}
\newtheorem{claim}[theorem]{Claim}
\newtheorem*{claim*}{Claim}

\newtheorem*{conjecture*}{Conjecture}

\newtheorem{assumption}[theorem]{Assumption}

\theoremstyle{remark}

\newtheorem*{example*}{Example}
\newtheorem{remark}[theorem]{Remark}
\newtheorem*{remark*}{Remark}

\newtheorem*{note*}{Note}
\newtheorem*{question*}{Question}

\newcommand{\norm}[1]{\left\lVert #1 \right\rVert}

\DeclareMathOperator{\supp}{supp}
\DeclareMathOperator{\id}{id}

\DeclareMathOperator{\Ad}{Ad}

\newcommand{\bbC}{\mathbb C} 

\numberwithin{equation}{section}

\newcounter{my_enumerate_counter}
\newcommand{\pushcounter}{\setcounter{my_enumerate_counter}{\value{enumi}}}
\newcommand{\popcounter}{\setcounter{enumi}{\value{my_enumerate_counter}}}

\usepackage{enumitem}

\begin{document}

\title[General Uniform Roe algebra rigidity]{General Uniform Roe algebra rigidity}%
\author[B. M. Braga]{Bruno M. Braga}
\address[B. M. Braga]{University of Virginia, 141 Cabell Drive, Kerchof Hall, P.O. Box 400137, Charlottesville, USA}
\email{demendoncabraga@gmail.com}
\urladdr{https://sites.google.com/site/demendoncabraga}

\author[I. Farah]{Ilijas Farah}
\address[I. Farah]{Department of Mathematics and Statistics,
York University,
4700 Keele Street,
Toronto, Ontario, Canada, M3J
1P3} 
\address{Matemati\v cki Institut SANU\\
Kneza Mihaila 36\\
11\,000 Beograd, p.p. 367\\
Serbia}
\email{ifarah@mathstat.yorku.ca}
\urladdr{http://www.math.yorku.ca/$\sim$ifarah}

\author[A. Vignati]{Alessandro Vignati}
\address[A. Vignati]{Universit\'e de Paris, 
Institut de Math\'ematiques de Jussieu - Paris Rive Gauche (IMJ-PRG)\\
B\^atiment Sophie Germain\\
8 Place Aur\'elie Nemours \\ 75013 Paris, France}
\email{ale.vignati@gmail.com}
\urladdr{http://www.automorph.net/avignati}

\subjclass[2010]{}
\date{\today}%
\begin{abstract}
We generalize all known results on rigidity of uniform Roe algebras to the setting of arbitrary uniformly locally finite coarse spaces. For instance, we show that isomorphism between uniform Roe algebras of uniformly locally finite coarse spaces whose uniform Roe algebras contain only compact ghost projections implies that the base spaces are coarsely equivalent. Moreover, if one of the spaces has property A, then the base spaces are bijectively coarsely equivalent. We also provide a characterization for the existence of an embedding onto hereditary subalgebra in terms of the underlying spaces. As an application, we partially answer a question of White and Willett about Cartan subalgebras of uniform Roe algebras.
\end{abstract}
\maketitle

\section{Introduction}\label{SectionIntro}
This paper deals with the rigidity problems for uniform Roe algebras associated to uniformly locally finite coarse spaces. Roe algebras, and their uniform versions, provide a bridge between operator algebras and large scale geometry. They were introduced by Roe in \cite{Roe1993}; since then, they were extensively studied because of their connections with geometric group theory, (higher) index theory, and consequently their applications to manifold topology and geometry (\cite{Roe1996}), and recently with topological phases of matter (\cite{Kubota2017}).

The original motivation for these rigidity problems stems from their connection with the Baum--Connes and Novikov conjectures (\cite{HigsonRoe1995,Yu2000}). In short, the rigidity problems ask what properties of a coarse space $X$ are remembered by the associated uniform Roe algebra. 
Functoriality of the uniform Roe algebras construction implies that a bijective coarse equivalence between spaces lifts to an isomorphism between the algebras, and that an injective coarse embedding lifts to an embedding into a hereditary \cstar-subalgebra. Similarly, in case of uniformly locally finite (u.l.f.) spaces, a coarse equivalence lifts to a stable isomorphism between the algebras, and a coarse embedding lifts to an embedding into a hereditary \cstar-subalgebra of the stabilization. 
 
\begin{problem}[Isomorphism rigidity problem]
Let $(X,\cE)$ and $(Y,\cE')$ be uniformly locally finite coarse spaces such that $\cstu(X,\cE)$ and $\cstu(Y,\cE')$ are isomorphic. Does it follow that $(X,\cE)$ and $(Y,\cE')$ are (bijectively) coarsely equivalent?
\end{problem}

\begin{problem}[Embedding rigidity problem]\label{ProbEmb}
Let $(X,\cE)$ and $(Y,\cE')$ be uniformly locally finite coarse spaces such that $\cstu(X,\cE)$ is isomorphic to a hereditary subalgebra of $\cstu(Y,\cE')$. Does it follow that $(X,\cE)$ (injectively) coarsely embed into $(Y,\cE')$? 
\end{problem}

In case we can prove the existence of a bijective coarse equivalence, or of an injective coarse embedding, we refer to \emph{strong} rigidity, while if we can only prove that the two spaces are coarsely equivalent (in case of isomorphisms of the uniform Roe algebras), or that $X$ coarsely embeds into $Y$ (in case of embeddings into a hereditary subalgebra), we refer to \emph{weak} rigidity\footnote{These should not be confused with the concepts of rigidity and superrigidity, see \cite[Remark 4.21]{WhiteWillett2017}.}.

 The first partial answer to the  isomorphism rigidity problem was given by \v{S}pakula and Willett in \cite{SpakulaWillett2013AdvMath}, where the authors showed that weak rigidity holds for metric spaces with Yu's property A. In \cite{BragaFarah2018} and \cite{BragaFarahVignati2019}, it was shown that, again in the case of isomorphism, strong rigidity holds for property A metric spaces, and that weak rigidity holds for metric spaces such that all sparse subspaces yield only compact ghost projections. This technical condition is satisfied by every uniformly locally finite metric space which satisfies the coarse Baum-Connes conjecture with coefficients (see \cite[Theorem 5.3]{BragaChungLi2019}).
 
In the context of not necessarily metrizable coarse spaces, the rigidity question remained mostly open. In fact, the only positive results so far assumed the spaces to be \emph{small} (this means that their coarse structure is generated by `not too many' subsets, see \cite[Definition 4.2]{BragaFarah2018}) or that the isomorphism is absolute under passing to forcing extensions of the set-theoretic universe (\cite[\S 9]{BragaFarah2018}). We substantially improve the status of this problem by generalizing rigidity phenomena to the general setting of uniformly locally finite coarse spaces. 

\begin{theorem}\label{TheoremMainIsomor}
Suppose that $(X,\mathcal E)$ and $(Y,\mathcal E')$ are uniformly locally finite coarse spaces such that $Y$ has property A. The following are equivalent.
\begin{enumerate}
\item\label{1.Main} The spaces $(X,\cE)$ and $(Y,\cE')$ are bijectively coarsely equivalent.
\item\label{1.1Main} There exists an isomorphism $\cstu(X,\cE)\to \cstu(Y,\cE')$ sending $\ell_\infty(X) $ to $\ell_\infty(Y)$.
\item \label{2.1Main} The $^*$-algebras $\cstu[X,\cE]$ and $\cstu[Y,\cE']$ are isomorphic.
\item \label{2.Main} The  \cstar	-algebras  $\cstu(X,\cE)$ and $\cstu(Y,\cE)$ are isomorphic.
\end{enumerate}
\end{theorem}

The implications \eqref{1.Main} $\Rightarrow$ \eqref{1.1Main} $\Rightarrow$ \eqref{2.Main} are trivial, and the equivalences \eqref{1.Main} $\Leftrightarrow$ \eqref{1.1Main} $\Leftrightarrow$ \eqref{2.1Main} are given by \cite[Theorem 8.1]{BragaFarah2018}. Those implications do not require property A. The implication \eqref{2.Main} $\Rightarrow$ \eqref{1.Main} is our contribution to Theorem \ref{TheoremMainIsomor}; it is the generalization of \cite[Theorem 1.11]{BragaFarahVignati2018} to not necessarily metrizable spaces, and its proof requires property A.

Regarding `weak rigidity', we generalize the findings of \cite{BragaFarahVignati2018} and \cite{BragaFarahVignati2019}.

\begin{theorem}\label{TheoremMainIsomor2}
Suppose that $(X,\mathcal E)$ and $(Y,\mathcal E')$ are uniformly locally finite coarse spaces such that $\cstu(X,\cE)$ and $\cstu(Y,\cE')$ have only compact ghost projections. Suppose that $\cstu(X,\mathcal E)$ and $\cstu(Y,\mathcal E')$ are isomorphic. Then $X$ and $Y$ are coarsely equivalent.
\end{theorem}

The following simple corollary of Theorem \ref{TheoremMainIsomor2} answers a question raised by White and Willett in \cite[Remark 3.4]{WhiteWillett2017} in the scenario of property A (see  Corollary \ref{CorWhillettWhite}  below).

\begin{corollary}\label{cor:metr}
Let $(X,\mathcal E)$ and $(Y,\mathcal E')$ be uniformly locally finite coarse spaces such that $\cstu(X,\cE)$ and $\cstu(Y,\cE')$ have only compact ghost projections, and let $\lambda$ be a cardinal. Suppose $\cstu(X,\cE)$ is isomorphic to $\cstu(Y,\cE')$. If $\mathcal E'$ is generated by a set of size $\lambda$, so is $\mathcal E$. In particular if $(Y,\mathcal E')$ is metrizable, so is $(X,\mathcal E)$.
\end{corollary}

Recall that two \cstar-algebras $A$ and $B$ are \emph{stably isomorphic} if $A\otimes \cK(H)$ and $B\otimes \cK(H)$ are isomorphic, where $\cK(H)$ is the ideal of compact operators on the separable, infinite-dimensional Hilbert space $H$.

\begin{theorem}\label{TheoremMainIsomorMorita}
Suppose that $(X,\mathcal E)$ and $(Y,\mathcal E')$ are uniformly locally finite coarse spaces and assume that for all $n\in\N$ all ghost projections in both $\cstu(X,\cE)\otimes M_n(\C)$ and $\cstu(Y,\cE')\otimes M_n(\C)$ are compact. The following are equivalent.
\begin{enumerate}
\item\label{1.MainMorita} The coarse spaces $(X,\cE)$ and $(Y,\cE') $ are coarsely equivalent.
\item\label{1.1MainMorita} The \cstar-algebras $\cstu(X,\cE)$ and $\cstu(Y,\cE')$ are Morita equivalent.
\item \label{2.MainMorita} The \cstar-algebras $\cstu(X,\cE)$ and $\cstu(Y,\cE')$ are stably isomorphic. 
\end{enumerate}
\end{theorem}

The implication \eqref{1.MainMorita} $\Rightarrow$ \eqref{1.1MainMorita} above is \cite[Theorem 4]{BrodzkiNibloWright2007}, and stable isomorphism and Morita equivalence agree for unital \cstar-algebras (this is both a special case of \cite[Theorem 1.2]{BrownGreenRieffel1977} and  our excuse for not defining Morita equivalence), so our contribution is the implication \eqref{2.MainMorita} $\Rightarrow$ \eqref{1.MainMorita}.


We then turn to the study of coarse embeddings and of Problem~\ref{ProbEmb}. The following `strong' and `weak' rigidity results generalize those proved in~\cite{BragaFarahVignati2019} for metric spaces. 

\begin{theorem}\label{TheoremMainEmbed}
Let $(X,\mathcal E)$ and $(Y,\mathcal E')$ be uniformly locally finite coarse spaces, and suppose that $(Y,\mathcal E')$ has property A. The following are equivalent. 
\begin{enumerate}
\item\label{ItemTheoremMainEmbed1} The coarse space $(X,\cE)$ injectively coarsely embeds into $(Y,\cE')$.
\item\label{ItemTheoremMainEmbed2} The \cstar-algebra $\cstu(X,\cE)$ is isomorphic to a hereditary \cstar-subalgebra of $\cstu(Y,\cE')$. 
\end{enumerate}
\end{theorem}

\begin{theorem}\label{thm:embednoghostproj}
Let $(X,\mathcal E)$ and $(Y,\mathcal E')$ be uniformly locally finite coarse spaces such that $\cstu(X,\cE)$ and $\cstu(Y,\cE')$ have only compact ghost projections. The following are equivalent. 
\begin{enumerate}
\item\label{Item1:thm:embednoghostproj} The coarse space $(X,\cE)$ coarsely embeds into $(Y,\cE')$.
\item\label{Item2:thm:embednoghostproj} The \cstar-algebra  $\cstu(X,\cE)$ is isomorphic to a hereditary \cstar-subalgebra of $\cstu(Y,\cE')\otimes \cK(H)$. 
\end{enumerate}
\end{theorem}

Some of the most studied coarse structures are the ones associated to the Cayley graph of a finitely generated group. In \S\ref{S.ex.nonmetrizable} \eqref{I.Group}, given a  discrete (not necessarily finitely generated) group $\Gamma$,  we define  a canonical uniformly locally finite coarse structure $\cE_\Gamma$ on it. In case $\Gamma$ is finitely generated, the coarse structure $ \cE_\Gamma$   coincides with the coarse structure associated with the Cayley graph metric on $\Gamma$. 	Generally, $(\Gamma,\cE_\Gamma)$ retains the desirable properties present in the finitely generated case such as the following analogue of \cite[Corollary~6.3]{SpakulaWillett2013AdvMath}.\footnote{The authors of \cite{SpakulaWillett2013AdvMath} also used a result of Whyte,  that the coarse equivalence and quasi-isometry agree in the category of nonamenable groups. It is not likely  that quasi-isometry has an analog in nonmetrizable  setting.} (Recall that a group is \cstar-exact if its reduced group \cstar-algebra is exact. For discrete groups, this is equivalent to $\cstu(\Gamma)$ being nuclear, see~\cite[Theorem 5.1.6]{BrownOzawa}.)

\begin{theorem}\label{T.Group} Let $\Gamma$ and $\Lambda$ be discrete and \cstar-exact groups. The following are equivalent. 
\begin{enumerate}
\item \label{1.T.Group} $\Gamma$ and $\Lambda$ are  coarsely equivalent. 
\item \label{2.T.Group}  $\cstu(\Gamma)$ and $\cstu(\Lambda)$ are Morita equivalent. 
\item \label{3.T.Group}  The reduced crossed products $\ell_\infty(\Gamma)\rtimes_r \Gamma$ and $\ell_\infty(\Lambda)\rtimes_r \Lambda$ are Morita equivalent. 
\end{enumerate}
\end{theorem}

\begin{proof} The equivalence of \eqref{2.T.Group} and \eqref{3.T.Group} is  Proposition~\ref{P.Crossed}\eqref{2.T.Group}; it does not use the assumption of \cstar-exactness. Since \cstar-exactness implies that all ghost projections are compact, the equivalence of \eqref{2.T.Group} and \eqref{1.T.Group} is Theorem~\ref{TheoremMainIsomorMorita}. 	
\end{proof}


The paper is structured as follows: \S\ref{SectionPrelim} introduces terminology and notation; in \S\ref{SectionPropA} we discuss property A and its equivalent reformulations applicable to general coarse spaces. In \S\ref{S.5} we introduce the main technical novelty used in the proofs of our results (see Lemma \ref{L.Claim}). Moreover, \S\ref{S.5}  deals with isomorphisms (and embeddings) between uniform Roe algebras which satisfy an extra rigidity condition (see Lemmas \ref{lemma:entourages} and \ref{lemma:coarse}), and in \S\ref{S.6} we show that property A (or just having only compact ghost projections in its uniform Roe algebra) implies that this rigidity condition holds. All our main results are proved in \S\ref{S.6}. In \S\ref{SectionAppli} they are applied to answer a question of White and Willett (see Corollary \ref{CorWhillettWhite}) about Cartan masas of uniform Roe algebras. We conclude by discussing uniform Roe algebras of  nonmetrizable coarse spaces in \S\ref{S.nonmetrizable}.

\section{Preliminaries}\label{SectionPrelim}
We record a few well-known facts about coarse spaces and their associated uniform Roe algebras.
Given a set $X$, $\cP(X)$ denotes its power set. The following should not be confused with the notion of a coarse topological space from \cite[Definition~5.0.1]{willett2019higher}, where the coarse structure is assumed to be countably generated. 

\subsection{Coarse spaces and geometry}\label{SubsectionCoarseSpace}
Given a set $X$, some $\mathcal E\subseteq\mathcal P(X\times X)$ is called a \emph{coarse structure on $X$} if 
\begin{enumerate}
\item the diagonal $\Delta_X=\{(x,x)\in X\times X\mid x\in X\}\in \cE$,
\item if $E\in \cE$, then $E^{-1}=\{(x,y)\in X\times X\mid (y,x)\in E\}\in \cE$,
\item if $E\in \cE$ and $F\subseteq E$, then $F\in \cE$, 
\item if $E,F\in \cE$, then $E\cup F\in \cE$, and
\item if $E,F\in \cE$, then $E\circ F\in\mathcal E$, where
\[
E\circ F=\{(x,z)\in X\times X\mid\exists y\in X,\ (x,y)\in E\wedge (y,z)\in F \}.
\]
\end{enumerate} 
 A coarse structure $(X,\cE)$ is \emph{connected} if it contains all finite subsets of $X\times X$ and $(X,\cE)$ is \emph{uniformly locally finite} (\emph{u.l.f.} from now on) if for all $E\in\mathcal E$ we have that 
\[
\sup_{x\in X}\{y\in X\mid (x,y)\in E\}<\infty.
\] 
Given $E\in \cE$ and $A\subseteq X$, we say that $A$ is \emph{$E$-bounded} if $A\times A\subseteq E$. 
 
Given sets $X$ and $Y$, a \emph{partial bijection between $X$ and $Y$} is a bijection $f:A\to B$ between $A\subseteq X$ and $B\subseteq Y$. For a proof of the following, see  \cite[Lemma 2.7(a)]{SkandalisTuYu2002} or \cite[Proposition 2.4]{BragaFarah2018}.

\begin{lemma}\label{lemma:splitting}
Let $(X,\mathcal E)$ be a u.l.f. coarse space. Then for every $E\in\mathcal E$ there is a finite partition $E=\bigsqcup_{i\leq n} E_i$ such that each $E_i$ is the graph of a partial bijection. \qed 
\end{lemma} 

If $\cF\subseteq\mathcal P(X\times X)$, we write $\langle \cF\rangle$ for the coarse structure generated by~$\cF$, that is, the smallest coarse structure on $X$ containing $\cF$. We say that $\cF$ \emph{generates}~$\langle \cF\rangle$. A coarse structure $\cE$ on $X$ is \emph{countably generated} if there exists a countable $\cF\subseteq \cP(X\times X)$ such that $\cE=\langle \cF\rangle$. 
 
If the set $X$ is equipped with a metric $d$, then $X$ has a canonical coarse structure, denoted by $\cE_d$, which is the coarse structure generated by the sets
\[
\{(x,y)\in X\times X\mid d(x,y)\leq n\}\text{, for } n\in\N.
\]
A coarse space $(X,\cE)$ is called \emph{metrizable} if there exists a metric $d$ on $X$ such that $\cE=\cE_d$. In this case, we say that the metric $d$ is \emph{compatible} with~$\cE$. 
Notably, a connected coarse structure $\mathcal E$ on $X$ is countably generated if and only if $\mathcal E=\mathcal E_d$ for some compatible metric $d$ by \cite[Theorem 2.55]{RoeBook}. More generally, we say that  $d\colon X\times X\to [0,\infty]$ is a \emph{$[0,\infty]$-valued metric} on $X$ if it is a symmetric function that separates the points of $X$ and satisfies the triangle inequality. Then \cite[Theorem 2.55]{RoeBook} gives the following.

\begin{lemma} \label{L.infty-valued-metric}
 	The following are equivalent for a coarse space $(X,\cE)$. 
 	\begin{enumerate}
 \item $\cE$ is generated by a $[0,\infty]$-valued metric.
 \item $\cE$ is countably generated. 
\item $\cE$ has a countable cofinal subset. \qed 
 	\end{enumerate}
\end{lemma}

Let $\lambda$ be a cardinal. We say that $\cE$ is $\lambda$-generated if $\lambda$ is the minimum size of a family $\mathcal F\subseteq\cE$ such that $\cE=\langle\cF\rangle$. Counter intuitively, if $\cE\subseteq\cE'$ and $\cE'$ is $\lambda$-generated, $\cE$ need not to be $\mu$-generated for some $\mu\leq\lambda$ (see Lemma~\ref{lemma:uf}).
 

\begin{definition}
Let $(X,\cE)$ and $(Y,\cE')$ be coarse spaces and $f,h\colon X\to Y$ be maps.
\begin{enumerate}
\item The map $f$ is \emph{coarse} if for all $E\in\cE$, 
\[
\Big\{(f(x),f(x'))\in Y\times Y\mid (x,x')\in E\Big\}\in \cE'.
\]
\item The maps $f$ and $h$ are \emph{close} if 
\[
\Big\{(f(x),h(x))\in Y\times Y\mid x\in X\Big\}\in \cE'.
\]
\item If $f$ is coarse and there is a coarse map $g\colon Y\to X$ such that $g\circ f$ is close to $\mathrm{Id}_X$ and $f\circ g$ is close to $\mathrm{Id}_Y$, we say that $f$ is a \emph{coarse equivalence}. In this case $(X,\cE)$ and $(Y,\cE')$ are said to be \emph{coarsely equivalent}. If the coarse equivalent is a bijection, the spaces are said to be  \emph{bijectively coarse equivalent}.
\item The map $f$ is a \emph{coarse embedding} if $f\colon X\to f(X)$ is a coarse equivalence.
\end{enumerate}
\end{definition}

Equivalently, coarse embeddings can be defined to be maps $f\colon (X,\cE)\to (Y,\cE')$ between coarse spaces which are coarse and \emph{expanding}, i.e., $f$ is such that, for all $E\in\cE'$, 
\[
\Big\{(x,x')\in X\times X\mid (f(x),f(x'))\in E\Big\}\in \cE.
\]
The existence of coarse embedding is monotonic with respect to the size of the smallest generating set, as witnessed by the following straightforward lemma whose proof we omit:
\begin{lemma}\label{LemmaSize}
Let $(X,\cE)$ and $(Y,\cE')$ be u.l.f. coarse spaces, and let $\lambda$ be a cardinal. Suppose that $X$ coarsely embeds into $Y$. If $\cE'$ is $\lambda$-generated, then $\cE$ is $\mu$-generated for some $\mu\leq\lambda$.\qed
\end{lemma}

\subsection{Uniform Roe algebras}\label{SubsectionURA}
Given a Hilbert space $H$, we denote the \cstar-algebra of bounded operators on $H$ by $\cB(H)$ and its ideal of compact operators by $\cK(H)$. If $X$ is a set, $\ell_2(X)$ denotes the Hilbert space of complex-valued $X$-indexed families with a distinguished orthonormal basis  $(\delta_x)_{x\in X}$. 
If $x,x'\in X$, the operator $e_{xx'}$ denotes the unique rank $1$ partial isometry sending $\delta_{x}$ to $\delta_{x'}$. For all operators $a\in\mathcal B(\ell_2(X))$ we have that
\[
\norm{e_{yy}ae_{xx}}=|\langle a\delta_x,\delta_y\rangle|.
\]
Let $(X,\cE)$ be a coarse space. Given an operator $a\in \cB(\ell_2(X))$, the support of $a$ is defined as
\[
\supp(a)=\{(x,y)\in X\times X\mid \langle a\delta_x, \delta_y\rangle\neq 0\}.
\]
We say that $a\in \cB(\ell_2(X))$ has \emph{$\cE$-bounded propagation} if $\supp(a)\in \cE$. The subset of all $\cE$-bounded operators forms a $^*$-algebra, whose operator norm closure is called the \emph{uniform Roe algebra of $(X,\cE)$}, denoted by $\cstu(X,\cE)$. When $\mathcal E$ is clear from the context, we simply write $\cstu(X)$ for $\cstu(X,\cE)$.
 
If $(X,\mathcal E)$ is a u.l.f. coarse space and $E\in\mathcal E$, we define an operator $v_E$ on~$\ell_2(X)$~by 
\[
\langle v_E\delta_{x},\delta_{x'}\rangle=\begin{cases}1&\text{ if }(x,x')\in E\\
0&\text{ otherwise.}
\end{cases}
\]
Notice that 
\[
v_E=\sum_{(x,x')\in E}e_{xx'}.
\]
Since $X$ is u.l.f.,  $v_E$ is bounded and it belongs to $\cstu(X)$. We call $v_E$ the \emph{partial translation associated to $E$}. We isolate the following for future reference. 

\begin{lemma}\label{lemma:struct}
Let $(X,\mathcal E)$ be a u.l.f. coarse space. Then 
\[
\mathcal E=\bigcup_{a\in\cstu(X,\mathcal E), \eps>0}\{(x,y)\in X\times X\mid \norm{e_{yy}ae_{xx}}>\eps\}.
\]
\end{lemma}
\begin{proof} 
The collection $\{v_E\in \cstu(X) \mid E\in\mathcal E\}$ witnesses the direct inclusion. For the converse inclusion, suppose that $E\subseteq X\times X$ is such that there are $a\in\cstu(X)$ and $\eps>0$ with $\norm{e_{x'x'}ae_{xx}}>\eps$	for all $(x,x')\in E$. Fix $b\in \cstu(X)$ of $\mathcal E$-bounded propagation such that $\|a-b\|<\eps$. Then $E$ is contained in the support of $b$. Since $\mathcal E$ is closed by subsets, we are done.
\end{proof}
If $\mathcal E\subseteq \mathcal E'$ are coarse structures on $X$, we say that $\mathcal E$ is a \emph{reduct} of $\mathcal E'$. We leave the proof of the following to the reader.

\begin{proposition}\label{PropInclusionCoarseStructure}
Let $Y\subseteq X$ be sets, and let $\mathcal E$ and $\mathcal E'$ be u.l.f. coarse structures on $Y$ and $X$, respectively. Then $\mathcal E\subseteq \mathcal E'$ if and only if the identity map induces an inclusion $\cstu(Y,\mathcal E)\subseteq\cstu(X,\mathcal E')$.
\end{proposition}

\begin{lemma}\label{L.reflection} 
Let $(X,\mathcal E)$ be a coarse space. Then $\cstu(X,\cE)=\bigcup_{\cF} \cstu(X,\cF)$, where $\cF$ ranges over all countably generated reducts of $\cE$.
\end{lemma}

\begin{proof} 
Only the direct inclusion requires a proof. For $a\in \cstu(X,\cE)$ and $m\in\N$ there is $b_m\in \cstu(X)$ with $\supp(b_m)\in \cE$ such that $\|a-b_m\|<1/m$. Hence $a$ belongs to $\cstu(X,\cF)$, where $\cF$ is the coarse structure generated by $\{\supp(b_m)\mid m\in \bbN\}$. 	
As $a$ was arbitrary, this completes the proof. 
\end{proof}

If $(X,\cE)$ is a connected u.l.f. coarse structure, Lemma~\ref{L.reflection} implies that $\cstu(X,\cE)$ is an inductive limit of a net of uniform Roe algebras associated with metric spaces. In case $X$ is not connected, Lemma~\ref{L.reflection} and Lemma~\ref{L.infty-valued-metric} together imply that $\cstu(X,\cE)$ is an inductive limit of a net of uniform Roe algebras associated to $[0,\infty]$-valued metrics on $X$. 

\subsection{Examples of nonmetrizable coarse spaces} 
\label{S.ex.nonmetrizable} 
For good measure, we give  few examples of nonmetrizable u.l.f. coarse spaces, after a simple observation.  Every connected u.l.f. coarse structure $\cE$ on an uncountable set $X$ is nonmetrizable. For this, fix a metric $d$  compatible with $\cE$, and let $x\in X$. Let $X_n=\{y\in X\mid d(x,y)\leq n\}$. As $\cE$ is connected, $X=\bigcup X_n$, so there is $n$ such that $X_n$ is uncountable. This contradicts the fact that $\cE$ is u.l.f.. 

\begin{enumerate}
\item The maximal coarse structure $\cE_{\max}$. 
On an infinite set $X$ consider the family of all subsets of $X\times X$ whose vertical and horizontal sections have uniformly bounded cardinalities. This is clearly a u.l.f. coarse structure. A diagonalization argument shows that it is not countably generated, even if $X$ is countable. It is also clear that every u.l.f. coarse structure on $X$ is included in this structure.   For this reason this structure is denoted $\cE_{\max}$. 

\item In \S\ref{S.nonmetrizable}  we show that there are $2^{2^{\aleph_0}}$ 
	coarsely inequivalent nonmetrizable coarse substructures of the standard metric structure on $\bbN$.  
\item\label{I.Group} Suppose that $\Gamma$ is a discrete group. For a finite $S\subseteq \Gamma$, let 
\begin{equation}\label{Ex.Group}
E_{S}=\{(g,h)\in \Gamma^2\mid gh^{-1}\in S\}. 
\end{equation}
Then the family $\cE_\Gamma$ of all subsets of $\Gamma^2$ included in $E_S$ for some finite $S\subseteq \Gamma$ is a coarse structure on $\Gamma$. All vertical and all horizontal sections of $E_{S}$ have cardinality at most $|S|$, and therefore this structure is u.l.f.. 
If $\Gamma$ is countable, then $\cE_\Gamma$  is countably generated and therefore metrizable. Since this space is clearly connected, it is nonmetrizable if $\Gamma$ is uncountable. 
\end{enumerate}

 If $\Gamma$ is finitely generated then $\cE_\Gamma$ coincides with the coarse structure associated with the Cayley graph metric on $\Gamma$. This is because if $S$ generates~$\Gamma$ then the sets~$E_{S^m}$, for $m\in \bbN$, are cofinal in $\cE_\Gamma$.  

\begin{proposition}\label{P.Group}
	For any uncountable cardinal $\kappa$ there exists an uncountable u.l.f. coarse space $(X,\cE)$ such that $X$ has cardinality $\kappa$ and for all $Y\subseteq X$ the induced coarse structure $(Y,\cE\cap Y\times Y)$ is metrizable if and only if $Y$ is countable. 
\end{proposition}

\begin{proof} Let $\Gamma$ be any group of cardinality $\kappa$, and let $X$ be $(\Gamma,\cE_\Gamma)$. 
	If $\Lambda$ is a subgroup of $\Gamma$ and $S\subseteq \Gamma$ is finite then clearly $E_S\cap (\Lambda\times \Lambda)=E_{S\cap \Lambda}\cap (\Lambda\times \Lambda)$ and $\cE_\Lambda=\{E\cap \Lambda^2\mid E\in \cE_\Gamma\}$, therefore $(\Lambda,\cE_\Lambda)$ is a coarse subspace of $(\Gamma,\cE_\Gamma)$. Every countable subspace of $\Gamma$ is included in a countable subgroup and therefore metrizable. Since $\Gamma$ is connected and u.l.f., no uncountable subspace is metrizable.  
\end{proof}

Interesting groups $\Gamma$ should result in interesting coarse structures $\cE_\Gamma$. There is a rich supply of uncountable groups with curious properties (see e.g., \cite{shelah1980on},~\cite{eklof2002almost},  \cite{magidor1994when}). 

By a result of Higson and Yu, if $\Gamma$ is a finitely generated discrete group then the uniform Roe algebra $\cstu(\Gamma)$ is isomorphic to the reduced crossed product $\ell_\infty(\Gamma) \rtimes_{r,\alpha} \Gamma$ associated to the translation action of $\Gamma$ on $\ell_\infty(\Gamma)$ (see e.g., \cite[Proposition~5.1.3]{BrownOzawa}) This observation extends to the coarse structure defined in \eqref{Ex.Group} above.

\begin{proposition}\label{P.Crossed} If $\Gamma$ is a discrete group equipped with the coarse structure defined in \eqref{Ex.Group}, then $\cstu(\Gamma)$  is isomorphic to $\ell_\infty(\Gamma)\rtimes_{r,\alpha} \Gamma$, where $\alpha$ is the translation action of $\Gamma$ on $\ell_\infty(\Gamma)$. 
\end{proposition}

\begin{proof} First we need to observe that $\cstu(\Gamma)$ is the \cstar-subalgebra of $\cB(\ell_2(\Gamma))$ generated by $\ell_\infty(\Gamma)$ and the unitaries $u_g$, for $g\in \Gamma$, defined by their action on the basis, $u_g \delta_h=\delta_{gh}$. Each $u_g$ has its support included in $E_{\{g\}}$, and it therefore suffices to prove that every operator with support included in $E_{S^m}$ for some $m$ is equal to a finite sum of the form $a=\sum_g a_g u_g$, with $a_g\in \ell_\infty(\Gamma)$.  If $\supp(a)\subseteq E_{S^m}$, then for $g\in S^m$ let $a_g=\Upsilon(a u_g^*)$ where $\Upsilon$ is the standard conditional expectation of $\cB(\ell_2(\Gamma))$ onto $\ell_\infty(\Gamma)$.\footnote{We recall the definition of conditional expectation in \S\ref{SectionAppli}.} Then clearly $a=\sum_{g\in S^m} a_g u_g$, as required. 

The remaining part of the proof  follows the proof in the case when $\Gamma$ is finitely generated (\cite[Proposition~5.1.3]{BrownOzawa}) closely. Let $u$ be the unitary on $\ell_2(\Gamma)\otimes \ell_2(\Gamma)$ defined by its action on the basis, $u(\delta_f\otimes \delta_g)=\delta_f \otimes \delta_{gf}$. 
A straightforward computation using the observation from the first paragraph shows that $u (\ell_\infty(\Gamma)\rtimes_{r,\alpha} \Gamma) u^*=1\cdot \bbC\otimes \cstu(\Gamma)\cong \cstu(\Gamma)$, completing the proof. 
\end{proof}

\subsection{Geometric assumptions}\label{ss:geom}

One of the geometric regularity conditions we will assume on the coarse structures of interest is property A. The following is  \cite[Definition~11.35]{RoeBook}.

\begin{definition} \label{Def.PropA} A coarse structure $(X,\cE)$ has \emph{property A} if for every $E\in \cE$ and every $m\geq 1$ there exist $F=F(E,m)\in \cE$ and a function $x\in X\mapsto \xi_x\in \ell_2(X)$ such that 
\begin{enumerate}
	\item $\|\xi_x\|=1$ for all $x\in X$, 
	\item $\{(x,x')\in X\times X\mid x'\in \supp(\xi_x)\}\subseteq F$, and 
	\item $(x,x')\in E$ implies $\|\xi_x-\xi_{x'}\|<1/m$, for all $x,x'\in X$. 
\end{enumerate}	
\end{definition}

 Nuclearity is arguably the most studied regularity property of \cstar-algebras. The following result can be found in  \cite[Theorem 5.5.7]{BrownOzawa} for metric spaces. However  it is not used in its proof, so we state it below for coarse spaces in general. 

\begin{theorem}\label{ThmNuclearPropA}
Let $(X,\cE)$ be a u.l.f. coarse space. Then $(X,\cE)$ has property~A if and only if $\cstu(X)$ is nuclear. \qed 
\end{theorem}

The format of property A most convenient for our approach is the so-called `Operator Norm Localization property' (ONL, see Definition \ref{DefiONL}). Property~A and the ONL are equivalent for metric spaces (see \cite[Theorem~4.1]{Sako2014}). In~\S\ref{SectionPropA}, we will prove that such an equivalence holds for general u.l.f. coarse spaces (see Corollary~\ref{CorPropA}). 

\begin{definition}\label{DefiONL}
A coarse space $(X,\mathcal E)$ has the \emph{Operator Norm Localization property} (\emph{ONL}) if for every entourage $E\in\mathcal E$ and $m>0$ there is an entourage $F=F(E,m)$ in $\cE$ such that for all $a\in\cstu(X)$ with $\supp(a)\subseteq E$, there is a unit vector $\xi\in \ell_2(X)$ for which $\supp(\xi)$ is $F$-bounded and it satisfies $\norm{a\xi}> (1-\frac{1}{m})\norm{a}$.
\end{definition}

A distinguished sort of operators is related to the failure of property A.  

\begin{definition}\label{DefiGhost}
An operator $a\in \cB(\ell_2(X))$ is called a \emph{ghost} if for all $\eps>0$ there exists a finite $A\subseteq X$ such that $|\langle a\delta_x,\delta_y\rangle|<\eps$ for all $(x,y)$ that lie outside of~$ A\times A$. 
\end{definition}

The following is \cite[Theorem 1.3]{RoeWillett2014}.\footnote{Although \cite[Theorem 1.3]{RoeWillett2014} is proven for metric spaces, the same proof works for infinite-valued metrics, i.e., for countably generated coarse spaces. See the remark following Lemma~\ref{L.reflection}.}

\begin{theorem}\label{ThmGhostProjCompPropA} 
Let $(X,\cE)$ be a u.l.f. coarse space and assume that $\cE$ is countably generated. Then $(X,\cE)$ has property A if and only if all ghost operators in $\cstu(X)$ are compact. \qed 
\end{theorem}

The second geometric property we will assume is weaker than property~A, and asks that all ghost projections in $\cstu(X)$ are compact. This is  (apparently) more restrictive than  a condition used in  \cite{BragaFarahVignati2018, BragaFarahVignati2019} (see \cite[Definition 1.6]{BragaFarahVignati2018}) that we now describe. If $(X,d)$ is a u.l.f. metric space, a sequence $\{X_n\}_n$ of nonempty finite subsets of $X$ is said to be \emph{sparse} if $d(X_n,X_m)\to\infty$ as $n+m\to\infty$ and $n\neq m$.  In the metric u.l.f. setting, sparse subspace are `everywhere': every sequence of nonempty disjoint finite subsets of the space under consideration has a sparse subsequence.  If $(X,\cE)$ is a u.l.f. coarse structure, it would be natural to call a subspace  sparse if for each $E\in \cE$ there was $n_0$ such that if $n,m\geq n_0$ with $n\neq m$ then $E\cap (X_n\times X_m)$ is empty.  However, it is not difficult to see that there are no sparse subsequences in $(\NN,\cE_{\max})$ defined in \S\ref{S.ex.nonmetrizable}. For this reason, sparse subspaces will not be further discussed in this paper.


\subsection{Graphs}
\label{S.Graphs}
Since two simple results from graph theory will be used in a critical place in our proofs, we recall the basic definitions. 
A pair $G=(V,E)$ is an (undirected and simple) \emph{graph} if $V$ is a set and $E\subseteq V\times V$ is such that $\Delta_V\cap E=\emptyset$ and $E=E^{-1}$. The elements of $V$ comprise the \emph{vertices} of $G$ and the elements of $E$ comprise the \emph{edges} of $G$. Two vertices $v$ and $w$ are \emph{adjacent} if $\{v,w\}$ is an edge. The degree of a vertex $v$, $\deg(v)$, is the cardinality of the set $\{w\in V\mid\{v,w\}\in E\}$. For $n\geq 1$, an $n$-\emph{colouring} of the vertices of $G$ is a partition of $V$ into $n$ pieces none of which contains two adjacent vertices. A straightforward argument using the Axiom of Choice shows that if $\deg(v)< n$ for all $v\in V$ then $G$ is $n$-colourable.

Another well-known graph-theoretic result that we will need is the infinitary version of Hall's Marriage Theorem. It states that if $\cF$ is an indexed family of finite sets such that every finite $F\subseteq \cF$ satisfies $|F|\leq |\bigcup F|$, then there exists an \emph{injective selector} $g\colon \cF\to \bigcup\cF$ (this means that $g(A)\in A$ for all $A\in \cF$ and $A\neq B$ implies $g(A)\neq g(B)$). The infinite version is a consequence of the finite Hall's Marriage Theorem and compactness. The proofs of these results can be found in any introductory text on graph theory. 

%
%
%
%

\section{Property A and reflection}\label{SectionPropA}
The main result of this section is that property A and ONL are equivalent for coarse spaces. Although this was already known (\cite{Sako2013}), we give an alternative shorter proof relying on a reflection argument quite different from the original.
 
Proposition~\ref{PropONLiffPropA} below is an easy application of the standard L\"owenheim--Sko\-lem/Black\-adar closing-off argument (see \cite[\S 7.3]{Fa:STCstar}, \cite[\S II.8.5]{Black:Operator}). 
We introduce some standard terminology, see \cite[\S 7.3]{Fa:STCstar}.\footnote{We will eventually prove that this terminology was not necessary; see Remark~\ref{R.silly}.} 
If $\cE$ is any uncountable set, a family $\cC$ of countable subsets of $\cE$ is a \emph{club} if it satisfies the following. 
\begin{enumerate}
\item ($\cC$ is \emph{cofinal}) For every countable $\cF\subseteq \cE$, there exists $\cF'\in \cC$ such that $\cF\subseteq \cF'$. 
\item ($\cC$ is \emph{$\sigma$-closed}) For every $\subseteq$-increasing sequence $\cF_n$, for $n\in \bbN$, in $\cC$, $\bigcup_n \cF_n\in \cC$. 	
\end{enumerate}
The following is an instance of a general definition, see \cite[Definition~7.3.1]{Fa:STCstar}. 

\begin{definition} 
Let $\bbP$ be a property of u.l.f. coarse spaces. We say that $\bbP$ \emph{reflects} if a u.l.f. coarse space $(X,\cE)$ has property $\bbP$ if and only if the family 
\[
\{\cF\subseteq \cE\mid (X,\langle \cF\rangle)\text{ has property $\bbP$}\}
\]
includes a club.
\end{definition}

A diagonalization argument shows that the intersection of any two clubs on $\cE$ is a club (see e.g., \cite[Proposition~6.2.9]{Fa:STCstar} with $\kappa=\aleph_1$, the least uncountable cardinal). This implies that if two properties of u.l.f. coarse substructures are equivalent for countably generated coarse structures and they both reflect, then they are equivalent for arbitrary coarse spaces.

\begin{proposition}\label{PropONLiffPropA}
Each of the following properties of u.l.f. coarse spaces reflects. 
\begin{enumerate}
\item\label{Item.PropA} Property~$A$.
\item\label{Item.PropONL} ONL.
\item\label{Item.Ghost} All ghosts in the associated uniform Roe algebra are compact.
\end{enumerate}
\end{proposition}

\begin{proof}
The proofs of \eqref{Item.PropA} and \eqref{Item.PropONL} use a similar closing-off argument. Lem\-ma~\ref{L.infty-valued-metric} implies that a coarse structure $\cF$ is countably generated if and only if it has a countable cofinal subset. This subset consists of the uniformities $\{(x,x')\in X\times X\mid d(x,x')<m\}$, for $m\in \bbN$, where $d$ is some $[0,\infty]$-valued metric that generates~$\cF$. 

Fix a u.l.f. coarse space $(X,\cE)$. For \eqref{Item.PropA}, note that Definition~\ref{Def.PropA} asserts that for every $E\in \cE$ and $m\geq 1$ there exist $F(E,m)\in \cE$ and a function $x\in X\mapsto \xi_x\in \ell_2(X)$ that does not depend on the coarse structure (apart from $F(E,m)$). This observation gives the converse implication: If $\cE$ has arbitrarily large countable structures closed under this operation, then~$\cE$ is closed under this operation as well. Moreover, this association is `weakly monotonic' in the sense that if $E'\subseteq E$ then $F(E,m)$ satisfies the requirements for $F(E',m)$.\footnote{We are not claiming that one can actually choose the function $(E,m)\mapsto F(E,m)$ to be monotonic.} 

For the direct implication, we need to show that if $(X,\cE)$ has property~A, then the set $\cC_A$ is cofinal and $\sigma$-closed, where
\[
\cC_A=\{\cF\subseteq \cE\mid (X,\langle \cF\rangle)\text{ has property A}\}.
\]
For a countable reduct $\cF$ of $\cE$ define $\tilde \cF$ as follows. Choose a countable cofinal subset $F_m$, for $m\in \bbN$, of $\cF$, and let $\tilde \cF$ be the coarse structure generated by $F(F_m,n)$, for $m\in \bbN$ and $n\in \bbN$. 
For a countably generated reduct $\cF_0\subseteq \mathcal E$ define $\cF_{n+1}=\tilde \cF_n$, for $n\geq 0$. Let $\cF'$ be the reduct generated by $\bigcup_n \cF_n$. It is clearly countably generated. By the weak monotonicity of the operation $F\mapsto F(E,m)$, the coarse structure $(X,\cF')$ has property~A, and therefore $\cC_A$ is cofinal. To show that $\cC_A$ is $\sigma$-closed, note that if $\cF=\bigcup_{n\in\N}\cF_n$, and $E\in\cF$, then there is $n$ such that $E\in\cF_n$. If each $\cF_n$ is closed by the operations $(E_n)\mapsto F(E,m)$, so is $\cF$. (Alternatively, note that $\cstu(X,\cF)$ is the direct limit of the nuclear \cstar-algebras $\cstu(X,\cF_n)$, hence if the latters are nuclear, so is $\cstu(X,\cF)$, meaning that $(X,\cF)$ has property A by Theorem~\ref{ThmNuclearPropA}).

To prove \eqref{Item.PropONL} note that the association $(E,m)\mapsto F(E,m)$ given by Definition~\ref{DefiONL} is weakly monotonic. Therefore the proof is a closing-off argument as in the proof of \eqref{Item.PropA}. The converse implication is again straightforward.

For \eqref{Item.Ghost}, note that Lemma~\ref{L.reflection} implies that there exists a noncompact ghost in $\cstu(X,\cE)$ if and only if there exists a noncompact ghost in $\cstu(X,\cF)$ for some countably generated reduct $\cF$ of $\cE$.
\end{proof}

The conclusion of the following corollary is a combination of known results and the equivalence \eqref{Item1} $\Leftrightarrow$ \eqref{Item2} announced in \cite{Sako2013}.

\begin{corollary}\label{CorPropA}
Let $(X,\mathcal E)$ be a u.l.f. coarse space. The following are equivalent. 
\begin{enumerate}
\item\label{Item1} $(X,\cE)$ has property A.
\item\label{Item2} $(X,\cE)$ has ONL.
\item\label{Item3} All ghost operators in $\cstu(X)$ are compact. 
\item\label{Item4} $\cstu(X)$ is nuclear. 
\item\label{Item1.1} Every countably generated reduct of $(X,\cE)$ has property A.
\item\label{Item2.1} Every countably generated reduct of $(X,\cE)$ has ONL.
\item\label{Item3.1} For every countably generated reduct $(X,\cF)$ of $(X,\cE)$, all ghost operators in $\cstu(X,\cF)$ are compact. 
\item\label{Item4.1} For every countably generated reduct $(X,\cF)$ of $(X,\cE)$, $\cstu(X,\cF)$ is nuclear. 
\end{enumerate}
\end{corollary}

\begin{proof} 
\eqref{Item1} and \eqref{Item4} are equivalent by 
\cite[Theorem 5.5.7]{BrownOzawa} (see the paragraph preceding Theorem~\ref{ThmNuclearPropA}). 
If $\cE$ is countably generated, then 
 \eqref{Item1} and \eqref{Item2} are equivalent by \cite[Theorem 5.1]{Sako2014}, and \eqref{Item1} and \eqref{Item4} are equivalent by \cite{RoeWillett2014} (see the paragraph preceding Theorem~\ref{ThmGhostProjCompPropA}). By Proposition~\ref{PropONLiffPropA}, each of \eqref{Item1}, \eqref{Item2}, and \eqref{Item3} reflects to countably generated reducts, and \eqref{Item1}--\eqref{Item4} are therefore equivalent.

By Proposition~\ref{PropONLiffPropA}\eqref{Item.Ghost}, \eqref{Item3.1} implies \eqref{Item3}. On the other hand, the property `all ghost operators in $\cstu(X,\cE)$ are compact' is preserved by taking reducts, by Proposition~\ref{PropInclusionCoarseStructure}, hence \eqref{Item3} implies \eqref{Item3.1}, and so the two are equivalent.

 Since \eqref{Item1}--\eqref{Item4} are equivalent in the case when $\cE$ is countably generated, \eqref{Item1.1}--\eqref{Item4.1} are equivalent to the first four items as well. 
 \end{proof}

\begin{remark} \label{R.silly} 
 We have proved that the closing-off arguments used in the proof of Proposition~\ref{PropONLiffPropA} were not needed. This is because by Corollary~\ref{CorPropA} \eqref{Item3.1} all relevant properties reflect to all countably generated reducts. 
 \end{remark}

\section{The `rigid' embedding scenario}
\label{S.5} 

In this section, we obtain coarse embeddability under the assumption that the embedding $\cstu(X)\to \cstu(Y)$ onto a hereditary subalgebra  satisfies a rigidity condition (see Lemmas \ref{lemma:entourages} and \ref{lemma:coarse}). In the following section we present geometric properties which guarantee that such rigidity condition holds. We emphasize that Lemma \ref{L.Claim} contains the main ideas which allow generalization of the existing rigidity proofs to the not necessarily metrizable case.

As in \cite{SpakulaWillett2013AdvMath},  we start by stating a lemma giving us a better grasp on the structure of an isomorphism (or embedding) $\Phi\colon\cstu(X)\to \cstu(Y)$. The following is precisely \cite[Lemma 6.1]{BragaFarahVignati2019}, but see also  \cite[Lemma~3.1]{SpakulaWillett2013AdvMath} (once again we warn the reader that although metrizability is present in the statement of \cite[Lemma 6.1]{BragaFarahVignati2019}, metrizability is not used in its proof).
 
\begin{lemma}\label{LemmaPhiStronglyContAndU} 
Let $(X,\cE)$ and $(Y,\cE')$ be u.l.f. coarse spaces. Let $\Phi\colon \cstu(X)\to \cstu(Y)$ be an isomorphism onto a hereditary subalgebra of $\cstu(Y)$. 
Then there exists an isometry $U\colon \ell_2(X)\to \ell_2(Y)$ such that
\[
\Phi(a)=UaU^*
\]
for all $a\in \cstu(X)$. 
Also, $\Psi(a)=U^*a U$ defines a completely positive and contractive map from $\cstu(Y)$ onto $\cstu(X)$ such that $\Psi\circ \Phi=\id_{\cstu(X)}$.\qed
\end{lemma}

%

The following will be our standing assumption throughout \S\ref{S.5} and \S\ref{S.6}. 

\begin{assumption}\label{Assumption}
Let $(X,\cE)$ and $(Y,\cE')$ be u.l.f. coarse spaces. Let $\Phi\colon \cstu(X)\to \cstu(Y)$ be an isomorphism onto a hereditary subalgebra of $\cstu(Y)$, and let $U\colon\ell_2(X)\to \ell_2(Y)$ be the isometry such that $\Phi(a)=\Ad U$, as  guaranteed by Lemma~\ref{LemmaPhiStronglyContAndU}.  Also define $\Psi\colon \cstu(Y)\to \cstu(X)$ by    
\[
\Psi(b)=U^*bU, 
\]
In other words,   $\Psi(b)=\Phi^{-1}(\Phi(1)b\Phi(1))$ for all $b\in \cstu(Y)$, where $\Phi^{-1}$ is defined only on the range of $\Phi$.
\end{assumption}

Note that, besides $\Psi\circ\Phi=\mathrm{Id}_{\cstu(X)}$, we have $\Phi\circ\Psi(b)=\Phi(1)b\Phi(1)$ for all $b\in \cstu(Y)$ (if $\Phi$ is an actual isomorphism, then clearly $\Phi\circ\Psi=\mathrm{Id}_{\cstu(Y)}$). In particular, \[\|a\Psi(b) \|=\|\Phi(a)b\Phi(1)\|\ \text{ and }\ \|\Psi(b)a \|=\|\Phi(1)b\Phi(a)\|\] for all  $a\in \cstu(X)$ and all $b\in \cstu(Y)$.  These equalities will be used throughout these notes with no further mention.


\begin{lemma} \label{L.XyYx} In the setting of Assumption~\ref{Assumption}: 
for  $x\in X$, $y\in Y$, and $\delta>0$ let 
\begin{enumerate}
\item \label{EqYdelta}
$Y_{x,\delta,\Phi}= \Big\{y\in Y\mid \norm{\Phi(e_{xx})e_{yy}\Phi(1)}>\delta\Big\}$ and  
\item  \label{EqXdelta} 
$X_{y,\delta,\Psi} =\Big\{x\in X\mid \norm{\Psi( e_{yy} )e_{xx}}>\delta\Big\}
$.
\end{enumerate}
Then  $y\in Y_{x,\delta,\Phi}$ if and only if $x\in X_{y,\delta,\Psi}$. 
Moreover, both 
$\sup_{x\in X}|Y_{x,\delta,\Phi}|$ and $\sup_{y\in Y}|X_{y,\delta,\Psi}|$ are finite.  
\end{lemma}

We will write $Y_{x,\delta}$ for $Y_{x,\delta,\Phi}$ and $X_{y,\delta}$ for $X_{y,\delta,\Psi}$ when $\Phi$ and $\Psi$ are clear from the context. 

\begin{proof}[Proof of Lemma \ref{L.XyYx}]
The first statement follows from the comments preceding the statement of this lemma.  For $x\in X$  and $y\in Y_{x,\delta}$ we have $\|\Phi(e_{xx})\delta_y\|\geq \|\Phi(e_{xx})e_{yy}\Phi(1)\|\geq \delta$, and therefore    $\sup_{x\in X} |Y_{x,\delta}|\leq \delta^{-2}$.  
An analogous argument shows that   $\sup_{y\in Y}|X_{y,\delta}|\leq \delta^{-2}$. 
\end{proof}

All the known proofs of (strong or weak) rigidity, starting with \cite{SpakulaWillett2013AdvMath}, proceed by analysis of the sets $Y_{x,\delta}$, along the following three steps.\footnote{A similar approach is used when dealing with strong or weak rigidity for embeddings from the existence of \cstar-embeddings onto hereditary subalgebras.}
\begin{enumerate}
\item\label{A.1} One finds $\delta>0$ such that $Y_{x,\delta}$ and $X_{y,\delta}$ are nonempty for all $x\in X$ and $y\in Y$.
\item \label{A.2} One proves that the sets $\bigcup_{x\in X}Y_{x,\delta}\times Y_{x,\delta}$ and $\bigcup_{y\in Y}X_{y,\delta}\times X_{y,\delta}$ are entourages (in $Y$ and $X$ respectively) for all $\delta>0$. (The metric version of this is \cite[Lemma 6.7]{WhiteWillett2017}).
\item \label{A.3} One uses \eqref{A.2} to show that, fixing $\delta>0$ satisfying \eqref{A.1}, any function $f\colon X\to Y$ such that $f(x)\in Y_{x,\delta}$ for all $x\in X$ is coarse and expanding.
\end{enumerate}
A proof of \eqref{A.1} typically uses a geometric assumption, such as the absence of noncompact ghost projections, or `rigidity' of the map $\Phi$ (see e.g., \cite[\S4]{BragaFarah2018}).
In the metrizable case, the proofs of \eqref{A.2} and \eqref{A.3} are obtained via a diagonalization argument, taking advantage of the fact that the coarse structure is countably generated. 
Our approach follows the steps \eqref{A.1}--\eqref{A.3}, but the absence of metrizability makes the proofs of \eqref{A.2} and \eqref{A.3} more difficult. We isolate a technical condition that guarantees the success of \eqref{A.2}~and~\eqref{A.3}. 
(Later, in \S\ref{section:geometry}, we will show how this technical condition follows from geometric assumptions on $X$ and $Y$.)

\begin{lemma} \label{L.Claim} In the setting of Assumption \ref{Assumption}: For all $\delta>0$ and $\eta>0$ there are finite partitions $\bigcup_{x\in X}Y_{x,\delta}\times Y_{x,\delta}=\bigsqcup_{i=0}^{n-1} B_i$, $X=\bigsqcup_{j=0}^{l-1} X_j$ and $Y=\bigsqcup_{\ell=0}^{k-1}Y_\ell$ such that for all $i$, $j$, and $\ell$ the following holds. 
\begin{enumerate}
 \item\label{1.Claim} All the horizontal and vertical sections of $B_i$ are at most singletons. 
\item \label{2.Claim} $(Y_{x,\delta}\times Y_{x,\delta})\cap B_i$ is at most a singleton for all $x$. 
	\item \label{3.Claim} If $x$ and $x'$ are distinct elements of $X_j$, then $Y_{x,\eta}$ and $Y_{x',\eta}$ are disjoint. 
	\item\label{4.Claim} If $y$ and $y'$ are distinct elements of $Y_\ell$, then $X_{y,\eta}$ and $X_{y',\eta}$ are disjoint.
\end{enumerate}

\end{lemma}

\begin{proof} We will use the fact that a graph in which the degrees of vertices are uniformly bounded is finitely chromatic (\S\ref{S.Graphs}). 
Let 
\[
A_\delta=\bigcup_{x\in X}Y_{x,\delta}\times Y_{x,\delta}.
\] 
Lemma~\ref{L.XyYx} implies that 
\[
\textstyle
n=\max\{\sup_{y\in Y} |X_{y,\delta}|, \sup_{x\in X} |Y_{x,\delta}|\}<\infty.
\]
For $y\in Y$ we have $\{y'\in Y\mid (y,y')\in A_\delta\}\subseteq \bigcup\{Y_{x,\delta}\mid x\in X_{y,\delta}\}$
and the right-hand side has cardinality at most $n^2$. 
Therefore the cardinalities of vertical and horizontal sections of $A_\delta$ are bounded by $n^2$. 
Define a graph with $A_\delta$ as the vertex-set by declaring that $(y,y')$ and $(y'',y''')$ are adjacent if and only if $y=y''$ or $y'=y'''$. 
The degrees of the vertices of this this graph are uniformly bounded, and therefore $A_\delta$ can be partitioned into finitely many pieces such that each of the pieces satisfies \eqref{1.Claim}.

To assure \eqref{2.Claim}, 
define a graph with $A_\delta$ as the vertex-set by declaring that the vertices $(y,y')$ and $(y'', y'')$ are adjacent if and only if they both belong to $Y_{x,\delta}\times Y_{x,\delta}$ for some $x$. By Lemma~\ref{L.XyYx}, the set of possible $x$ is included in $X_{y,\delta}$ and therefore of cardinality at most $n$. As for every $x\in X$ we have $|Y_{x,\delta}\times Y_{x,\delta}|\leq n^2$, the degrees of the vertices of this graph are uniformly bounded by $n^3$ and $A_\delta$ can be partitioned into finitely many pieces each of which satisfies \eqref{2.Claim}. The pieces of the joint refinement of these two partitions, $A_\delta=\bigcup_{i=0}^{k-1} B_i$, satisfy \eqref{1.Claim} and \eqref{2.Claim}. 

We now want to find a partition of $X$ satisfying \eqref{3.Claim}. Define a graph on $X$ by making $x$ and $x'$ adjacent if $Y_{x,\eta}\cap Y_{x',\eta}\neq\emptyset$. All vertices adjacent to a fixed $x$ belong to the set $\bigcup\{X_{y,\eta}\mid y\in Y_{x,\eta}\}$. Therefore Lemma \ref{L.XyYx} implies that this graph has uniformly bounded degree, and we can find a colouring $X=\bigsqcup_{i=0}^{l-1} X_{i}$ where each $X_{i}$ is monochromatic. Clearly, the pieces of this partition satisfy \eqref{3.Claim}. A partition of $Y$ satisfying \eqref{4.Claim} is obtained in the same exact way.
\end{proof}

The following auxiliary lemma can be compared with \cite[Lemma~6.5]{BragaFarahVignati2018}.


\begin{lemma}\label{LemmaRank1Inequality}
If $b,c$, and $v$ are rank 1 operators on a Hilbert space, then 
\[
\|v\|\|bvc\|= \|bv\|\|vc\|.
\] 
\end{lemma}

\begin{proof} 
By replacing $v$ with $v/\|v\|$, we may assume that it is a partial isometry, hence $v=vv^*v$. We then have 
\[
\|vc\|^2=\|c^*v^*vc\|=\|c^*v^*vv^*vc\|=\|v^*vc\|^2,
\]
 hence  $\|vc\|=\|v^*vc\|$. 
 If $d$ is a nonzero rank 1 operator, then the source projection of $d$, $p=d^*d\|d\|^{-1} $, is of rank 1 and  $\|d\|=\|d\xi\|$, where $\xi$ is any unit vector in the range of $p$. 

We may assume that $bv\neq 0$ and $vc\neq 0$. Therefore each one of $bv$, $vc$, and $bvc$ has rank 1, and these three operators have the same source projection. Let $\xi$ be a unit vector such that $\|vc\|=\|vc\xi\|$. 
Then $\|vc\|=\|v^*vc\xi\|$. 
The vector $\eta=v^*vc\xi$ satisfies $v^*v\eta=\eta$, and  
\[
\|bvc\|=\|bvv^*vc\|=\|bvv^*vc\xi\|=\|bvv^* \eta\|=\|bv\|\|\eta\|=\|bv\|\|vc\|. 
\]
This proves the first equality, and the second follows by $\|vc\|=\|v^*vc\|$ and 
$\|bv\|=\|bvv^*\|$ proved in the analogous way. 
	\end{proof}

We are now ready to prove the main results of this section (the further hypotheses assumed in Lemmas \ref{lemma:entourages} and \ref{lemma:coarse} below are what we called ``rigidity conditions'' in the introduction).

\begin{lemma}\label{lemma:entourages}
In the setting of Assumption \ref{Assumption}: Further suppose that for all $\eps>0$ there is $\eta>0$ such that for all $y\in Y$ we have that
\[
\norm{(1-\chi_{X_{y,\eta}})\Psi(e_{yy})}<\eps.
\]
Then for all $\delta>0$, the set $A_\delta=\bigcup_{x\in X}(Y_{x,\delta}\times Y_{x,\delta})$ is an entourage.
\end{lemma}

\begin{proof}
Fix $\delta>0$ and pick $\eta\in (0,\delta)$  such that $\norm{(1-\chi_{X_{y,\eta}})\Psi(e_{yy})}<\frac{\delta^2}{2}$ for all $y\in Y$. Applying Lemma~\ref{L.Claim} to $\delta$ and $\eta$, we can find partitions $A_\delta=\bigsqcup_{i=0}^{k-1} B_i$ and $X=\bigsqcup_{j=0}^{l-1} X_j$ such that all horizontal and vertical sections of each $B_i$ are at most singletons, $(Y_{x,\delta}\cap Y_{x,\delta}) \cap B_i$ is at most a singleton for all $x$, and if $x$ and $x'$ are distinct elements of $X_j$ then $Y_{x,\eta}\cap Y_{x',\eta}=\emptyset$. 

For $i<k$ and $j<l$ let 
\[
F_{i,j}=B_i\cap \bigcup_{x\in X_j} (Y_{x,\delta}\times Y_{x,\delta}). 
\]
 Since $A_\delta=\bigsqcup_{i=0}^{k-1}\bigsqcup_{j=0}^{l-1} F_{i,j}$, it suffices to show that each $F_{i,j}$ is an entourage. 
 
Fix $i$ and $j$.   Let $(y,y')\in F_{i,j}$ and let $x\in X_j$ be such that $y\in Y_{x,\delta}$ and $y'\in Y_{x,\delta}$. Such an $x$ is unique:   by condition \eqref{3.Claim} of Lemma~\ref{L.Claim}, $Y_{x,\eta}$ and $Y_{x',\eta}$ are disjoint for distinct $x,x'\in X_j$. As $\eta<\delta$, we have that  $ Y_{x,\delta}\cap Y_{x',\delta}=\emptyset$ for all distinct $x,x'\in X_j$.

 As $y$ and $y'$ both belong to $Y_{x,\delta}$,  Lemma~\ref{LemmaRank1Inequality} applied to $v=\Phi( e_{xx})$, $b=e_{y'y'}$, and $c=e_{yy}\Phi(1)$ gives that
\[
\norm{e_{y'y'}\Phi( e_{xx}) e_{yy}\Phi(1)}=\| e_{y'y'}\Phi( e_{xx})\|\| \Phi( e_{xx}) e_{yy}\Phi(1)\|
\geq\delta^2,
\]
hence 
\[
\norm{e_{y'y'} \Phi( \chi_{X_j}) e_{yy}\Phi(1)}\geq\delta^2 -\norm{e_{y'y'}\Phi(\chi_{X_j\setminus \{x\}}) e_{yy}\Phi(1)}.
\]
Since $y\in Y_{x,\delta}$, we have $y\notin Y_{x',\eta}$ for all $x'\in X_j\setminus \{x\}$, and therefore $X_j\setminus \{x\}\subseteq X\setminus X_{\eta,y}$. Hence 
$\chi_{X_j\setminus\{x\}}=\chi_{X_j\setminus\{x\}}(1-\chi_{X_{y,\eta}})$.
By  the choice of~$\eta$, this implies that 
\[
\norm{e_{y'y'} \Phi(\chi_{X_j\setminus \{x\}} ) e_{yy}\Phi(1)}\leq  \norm{(1-\chi_{X_{y,\eta}})\Psi(e_{yy})}<\frac{\delta^2}2, 
\]
and we conclude that 
$
\norm{e_{y'y'} \Phi( \chi_{X_j}) e_{yy}\Phi(1)}\geq\frac{\delta^2}{2}$.
Therefore 
\[
F_{i,j}\subseteq\{(y,y')\in Y^2\mid \norm{e_{y'y'} \Phi(\chi_{X_j} )e_{yy}}>\eps\}.
\] 
By Lemma~\ref{lemma:struct}, this concludes the proof.
\end{proof}

\begin{lemma}\label{lemma:coarse}
In the setting of Assumption \ref{Assumption}: Further suppose that for all $\eps>0$ there is $\eta>0$ such that for all $y\in Y$ and all $x\in X$ we have that
\[ \norm{(1-\chi_{X_{y,\eta}})\Psi( e_{yy})}<\eps
\] 
and
\[
\norm{(1-\chi_{Y_{x,\eta}})\Phi(e_{xx} )}<\eps
\]
Let $\delta>0$ and let $f\colon X\to Y$ be a function such that $f(x)\in Y_{x,\delta}$ for all $x\in X$. Then $f$ is coarse and expanding.
\end{lemma}

\begin{proof}
We first prove that $f$ is coarse. Let $E\in\mathcal E$ be an entourage, and let us show that 
\[
(f\times f)[E]=\{(f(x),f(x'))\in Y\times Y\mid (x,x')\in E\}
\]
 is an entourage in $Y\times Y$. Note that if $E=\bigsqcup_{i\leq n} E_i$ then 
 \[
 (f\times f)[E]=\bigsqcup_{i\leq n} (f\times f)[E_i].
 \] 
 Since $X$ is u.l.f., Lemma \ref{lemma:splitting} allows us to assume that all vertical and horizontal sections of $E$ are at most singletons. By our hypotheses there is a positive $\eta<\delta$ such that for all $y\in Y$
\[
\norm{(1-\chi_{X_{y,\eta}})\Psi(e_{yy})} <\frac{\delta^2}{2}. 
\]
By Lemma~\ref{L.Claim}, we can partition $X$ into finitely many pieces $X=\bigsqcup_{j=0}^{l-1} X_j$ such that if $x$ and $x'$ are distinct elements of $X_j$ then $Y_{x,\eta}$ and $Y_{x',\eta}$ are disjoint. For each $j<l$, let $E_j=E\cap (X_j\times X)$. It suffices to show that $(f\times f)[E_j]$ is an entourage for each $j< l$. The strategy is similar to the one used in Lemma~\ref{lemma:entourages}: we want to find $\eps>0$ and and operator $T\in\cstu(Y)$ such that 
\[
(f\times f)[E_j]\subseteq\{(y,y')\in Y^2\mid \norm{e_{y'y'}Te_{yy}}>\eps\}.
\]
Fix $j< l$ and let $v_j\in\cstu(X)$ be the partial translation associated to $E_j$ (as in the paragraph before  Lemma~\ref{lemma:struct}). Let $(y,y')\in (f\times f)[E_j]$, and fix $x\in X_j$ and $x'\in X$ such that $(x,x')\in E_j$, $f(x)=y$ and $f(x')=y'$. Since $f(x)\in Y_{x,\delta}$ for all $x\in X$, by applying Lemma~\ref{LemmaRank1Inequality}, we have 
 \[
 \norm{e_{y'y'} \Phi(e_{xx'} )e_{yy}\Phi(1)}= 
 \| e_{y'y'} \Phi( e_{xx'} ) \| \| \Phi(e_{xx'}) e_{yy} \Phi(1)\|\geq  \delta^2.
 \]
Let $A_x=\{z\in X\mid z\neq x\text{ and } \exists z'\in X \text{ with }(z,z')\in E_j\}$, and note that 
$
(v_j-e_{xx'})\chi_{A_x}=v_j-e_{xx'}$ and that 
 if $z\in A_x$, then $z\in X_j$. As $y\in Y_{x,\eta}$,  the choice of the partition $X_j$ implies that  $y\notin Y_{z,\eta}$, and therefore $z\notin X_{y,\eta}$. This shows that $A_x\subseteq X\setminus X_{y,\eta}$, hence 
\[
v_j-e_{xx'}=(v_j-e_{xx'})\chi_{A_x}=(v_j-e_{xx'})(1-\chi_{X_{y,\eta}}).
\]
By our choice of $\eta$, this implies that
\[
\norm{e_{y'y'} \Phi(v_j-e_{xx'} )e_{yy}\Phi(1)}\leq \norm{ (v_j-e_{xx'}) (1-\chi_{X_{y,\eta}}) \Psi(e_{yy})}\leq \frac{\delta^2}{2}.
\]
Hence, if $(y,y')\in (f\times f)[E_j]$ then   
\begin{multline*}
\norm{e_{y'y'}\Phi(v_j )e_{yy}}\geq\norm{e_{y'y'}\Phi(v_j )e_{yy}\Phi(1)}\\
\geq \norm{e_{y'y'}\Phi(  e_{xx'}) e_{yy}\Phi(1)}-\norm{e_{y'y'} \Phi(v_j-e_{xx'}) e_{yy}\Phi(1)}\geq\frac{\delta^2}{2}.
\end{multline*}
By Lemma~\ref{lemma:struct}, this finishes the proof that $f$ is coarse.

We now prove that $f$ is expanding. We reboot the notation in this proof and fix   $E\in\cE'$. By partitioning it into finitely many pieces, we may assume that all vertical and horizontal sections of $E$ are singletons. We want to show that 
\[
\{(x,x')\in X^2\mid (f(x),f(x'))\in E\}
\]
is an entourage. By the assumptions, we can fix  $\eta>0$ small enough to have  
\[
\norm{(1-\chi_{Y_{x,\eta}}) \Phi( e_{xx} )}<\frac{\delta^2}{4}
\]
 for all $x\in X$. By Lemma~\ref{L.Claim}, we can find a finite partition of $Y=\bigsqcup_{j=0}^{k-1}Y_j$ such that for all $j<k$, the sets $X_{y,\eta}$ and $X_{y',\eta}$ are disjoint for distinct $y,y'\in Y_j$. For $j< k$, let
\[
E_j=\{(y,y')\in E\mid y\in Y_j\}. 
\]
It suffices to show that each 
\[
A_j=\{(x,x')\in X^2\mid (f(x),f(x'))\in E_j\}
\]
 is an entourage. Let $v_{j}$ the partial translation associated to $E_j$. We claim that 
\[
(f\times f)^{-1}[E_j]\subseteq \{(z,z')\in X^2\mid \norm{e_{z'z'} \Psi(v_j ) e_{zz}}>\delta^2/2\}.
\]
This will suffice by Lemma~\ref{lemma:struct}, since $\Psi( v_j )\in \cstu(X)$. Pick $(x,x')\in A_j$. As $f(x)\in Y_{x,\delta}$, we have $x\in X_{f(x),\delta}$, and similarly $x'\in Y_{f(x'),\delta}$, hence Lemma~\ref{LemmaRank1Inequality}  implies that
\begin{equation}\label{Eq.xx'}
 \| e_{x'x'} \Psi (e_{f(x) f(x')} ) e_{xx}\|
\geq \| e_{x'x'} \Psi( e_{f(x) f(x')} ) \|\|\Psi( e_{f(x) f(x')} ) e_{xx}\|
\geq\delta^2.
\end{equation}
With $B_x=\{y\in Y\mid(\exists y'\in Y) (y,y')\in E_j\}\setminus\{(f(x),f(x'))\}$ we have 
\[
v_{j}-e_{f(x)f(x')}=(v_{j}-e_{f(x)f(x')})\chi_{B_x}.
\]
 Moreover, as for all distinct $y$ and $y'$ in $Y_j$ the sets $X_{y,\eta}$ and~$X_{y',\eta}$ are disjoint, we have $x\notin X_{y,\eta}$ for all $y\in B_x$, meaning that $B_x\subseteq Y\setminus Y_{x,\eta}$. This shows that 
\[
v_{j}-e_{f(x)f(x')}=(v_{j}-e_{f(x)f(x')})(1-\chi_{Y_{x,\eta}}).
\]
Hence,   our  choice of $\eta$, implies that  
\begin{align*}
\norm{\Psi(v_{j}-e_{f(x)f(x')} )e_{xx}}
&=\norm{\Phi(1)( v_{j}-e_{f(x)f(x')})  \Phi(e_{xx})}\\
& \leq \norm{  ( v_{j}-e_{f(x)f(x')}) (1-\chi_{Y_{x,\eta}})\Phi( e_{xx})}\\
& \leq\frac{\delta^2}{2}.
\end{align*}
By this and \eqref{Eq.xx'}, for all $(x,x')\in A_j$ we have $\norm{e_{x'x'} \Psi( v_{j} ) e_{xx}}\geq\frac{\delta^2}{2}$, or, in other words, the thesis.
\end{proof}

\section{Obtaining rigidity and coarse equivalence}\label{section:geometry}
\label{S.6} 
In this section we show how geometric assumptions on the spaces $X$ and~$Y$ are  sufficient to imply the assumptions of Lemmas~\ref{lemma:entourages} and~\ref{lemma:coarse}. As a consequence, we prove all our main results.

\subsection{The case of property A}
  An operator between two uniform Roe algebras is \emph{compact-preserving} if it sends compact operators to compact operators. We start by a relative to  \cite[Lemma~3.3]{BragaVignati2019}.

\begin{lemma} \label{L:AA} 
Let $(X,\cE)$ and $(Y,\cE')$ be u.l.f. coarse spaces, and let $\Phi\colon \ell_\infty(X)\to \cstu(Y)$ be a strongly continuous bounded linear map. Then for every $b$ in $\cK(\ell_2(Y))$ and every $\e>0$ there exists a finite $F\subseteq X$ such that for all contractions $a\in \ell_\infty(X\setminus F)$ we have $\max\{\norm{b\Phi(a)},\norm{\Phi(a)b}\}<\e$. 
\end{lemma}

\begin{proof} Assume otherwise. 

Recursively find contractions $a_n$, for $n\in \bbN$, in $\ell_\infty(X)$ and finite $F_n\subseteq X$ such that $F_n\subseteq F_{n+1}$,  $\max(\|b\Phi(a_n)\|, \|\Phi(a_n) b)\geq \eps/2$,  and $\supp(a_n)\subseteq F_n$. This is possible since by the strong continuity of $\Phi$, $\|\Phi(a)b\|> \eps/2$  
	implies $\|\Phi(a \chi_G)b\|>\eps/2$ for some finite $G\subseteq X$. 
		Then $a_n\to 0$ in the strong operator topology, hence $\Phi(a_n)\to 0$ in the strong operator topology. Since $b$ is compact, both $\Phi(a_n)b$ and $b\Phi(a_n)$  converge to $0$ in norm; contradiction. 
\end{proof}

The metric version of the following is \cite[Lemma 6.6]{BragaFarahVignati2019} (cf. \cite[Lemma~6.6]{WhiteWillett2017}).

\begin{lemma}\label{LemmaBla1}
Let $(X,\mathcal E)$ and $(Y,\mathcal E')$ be u.l.f. coarse spaces, assume that $(Y,\mathcal E')$ has property $A$, and let $\Phi\colon\cstu(X)\to\cstu(Y)$ be a compact-preserving strongly continuous bounded linear map. Let $a\in \cstu(Y)$ and $\eps,\delta>0$. Let $(A_n)_n$ and $(B_n)_n$ be sequences of finite subsets of $X$ and $Y$, respectively, such that at least one of them is a sequence of disjoint sets. If $\norm{\Phi(\chi_{A_n})\chi_{B_n}a}>\varepsilon$ for all $n\in\NN$, then there exists $k>0$ such that for infinitely many $n\in \bbN$ some $D_n\subseteq B_n$ satisfies $|D_n|\leq k$ and 
\[
\norm{\Phi(\chi_{A_n})\chi_{D_n}a}>\varepsilon-\delta.
\]
\end{lemma}

\begin{proof}
Fix $\varepsilon$, $\delta$, $(A_n)_n$ and $(B_n)_n$ as in the hypotheses of the lemma. We first show that we may assume that both $(A_n)_n$ and $(B_n)_n$ are sequences of disjoint sets.

First consider the case when $(B_n)_n$ is a sequence of disjoint sets.
Since $\Phi$ is compact-preserving, we have $\lim_n\|\Phi(\chi_G)\chi_{B_n}\|=0$ for every finite $G\subseteq X$. Hence, by going to a subsequence and shrinking the $A_n$'s, we can assume that the $A_n$'s are also disjoint. 
The argument in the case when $(A_n)_n$ is a sequence of disjoint sets  is analogous and therefore omitted.

 For all $n\in\NN$ we have
 \[
 \norm{\Phi(\chi_{A_n})\chi_{B_n}a}>\varepsilon-\frac{\delta}{32}.
 \]
Using that $\Phi$ is compact-preserving once again, pick a sequence $(C_n)_n$ of finite subsets of $Y$ such that such that 
\[
\norm{(1-\chi_{C_n})\Phi(\chi_{A_n})}<2^{-n-1}.
\]
By applying Lemma \ref{L:AA}, we can assume that the $C_n$'s are disjoint, and that our subsequence satisfies $\|\Phi(\chi_{A_n})\chi_{B_m}\|\leq 2^{-n-m}$ for all $n\neq m$. Let $A=\bigcup_n A_n$ and $B=\bigcup_n B_n$. Since each $\Phi(\chi_{A_n})\chi_{B_m}$ is compact, absolute convergence implies that 
$
\Phi(\chi_A)\chi_B-\sum_{n\in\N}\chi_{C_n}\Phi(\chi_{A_n})\chi_{B_n}$, 
is compact and therefore
\[
c=\sum_{n\in\N}\chi_{C_n}\Phi(\chi_{A_n})\chi_{B_n}
\]
belongs to $\cstu(Y)$.

Pick $d\in\cstu[Y]$ with $\norm{c-d}<\delta/(64\|a\|)$ and let $F=\supp(d)$. 
 For each $n\in\N$ let $d_n=\chi_{C_n}d\chi_{B_n}$. Then $\supp(d_n)=F\cap (C_n\times B_n)\subseteq F$.  Fix $n_0\in\N$ such that $2^{-n_0-1}<\delta/(64\|a\|)$. Then 
\begin{align*}
\norm{d_na-\Phi(\chi_{A_n})\chi_{B_n}a}&\leq \norm{d_na-\chi_{C_n}\Phi(\chi_{A_n})\chi_{B_n}a}\\&
\ \ \ \ +\norm{(1-\chi_{C_n})\Phi(\chi_{A_n})\chi_{B_n}a}\\
&<\delta/32.
\end{align*}
for $n\geq n_0$. In particular $\norm{d_na}>\varepsilon-\delta/16$. Fix $b\in \cstu[Y]$ such that $\|a-b\|<\delta/(16)$. Then $\norm{d_nb}>\varepsilon-\delta/8$ for all $n\geq n_0$.

Fix $\gamma\in (0,1)$  such that $\gamma(\eps-\delta/8)> \eps-\delta/4$ and let $F'\in\mathcal E'$ be given by the definition of ONL for $F=\supp(d)$ and $\gamma$. Fix $n\in \bbN$ for a moment. Since  $\supp(d_nb)\subseteq F$,  there exists a unit vector $\xi_n\in \ell_2(Y)$ with $\supp(\xi_n)^2\subseteq F'$ and such that $\norm{d_nb\xi_n}>\eps-\delta/4$. Let $D_n=\supp(b\xi_n)$ and note that 
\[
D_n\subseteq \Big\{y\in Y\mid \exists y'\in \supp(\xi_n) , \ (y,y')\in \supp(b)\Big\}.
\]
Since $(Y,\cE')$ is u.l.f., there is $k$ such that $\sup_{n\in\N}|D_n|<k$.

As $\norm{d_nb\xi_n}>\eps-\delta/4$, we have that $\norm{d_n\chi_{D_n}b}>\eps-\delta/4$, which implies that $\|d_n\chi_{D_n}a\|> \eps-\delta/2$, for all $n\in\N$. Since $\norm{d_na-\Phi(\chi_{A_n})\chi_{B_n}a}<\delta/32$, we have that 
\[\norm{\Phi(\chi_{A_n})\chi_{B_n\cap D_n}a}\geq \|d_n\chi_{D_n}a\|-\|d_n\chi_{D_n}a-\Phi(\chi_{A_n})\chi_{B_n \cap D_n}a\|>\varepsilon-\delta\] for all $n\geq n_0$. This concludes the proof. 
\end{proof}

In the following lemma we obtain the conclusion analogous to that in Lemma~\ref{LemmaBla1} from weaker assumptions. 

\begin{lemma}\label{LemmaBla2}
Let $(X,\mathcal E)$ and $(Y,\mathcal E')$ be u.l.f. coarse spaces, assume that $(Y,\mathcal E')$ has property $A$, and let $\Phi\colon\cstu(X)\to\cstu(Y)$ be a compact-preserving strongly continuous bounded linear map. 
Suppose that $a\in \cstu(Y)$, $\eps>0$, and $\delta>0$. Then there is $k>0$ such that for all finite
$A\subseteq X$ and $B \subseteq Y$, if $\|\Phi(\chi_A)\chi_Ba\|>\eps$ then there exists $D\subseteq B$ with $|D|<k$ such that
$\|\Phi(\chi_{A})\chi_Da\|\geq \varepsilon/2-\delta$.
\end{lemma}
\begin{proof}
Assume otherwise, and let $\delta$ and $\varepsilon$ and sequences $(A_n)_n$ and $(B_n)_n$ of finite subsets of $X$ and $Y$, respectively, such that $\norm{\Phi(\chi_{A_n})\chi_{B_n}a}>\varepsilon$, for all $n\in\N$, and $\norm{\Phi(\chi_{A_n})\chi_{D}a}<\varepsilon/2-\delta$ for all $D\subseteq B_n$ with $|D|<n$. 

By compactness of $2^{X}$, the sequence $(B_n)_n$ has a subnet which converges pointwise to some $L\subseteq X$.\footnote{Since $\bigcup_n B_n$ is countable one easily obtains a convergent subsequence, but this fact will not be used since it does not help simplify the proof.} In particular, given a finite $F
\subseteq X$ and $\ell\in\N$, there exists $n>\ell$ such that 
\[B_{n}\cap F =L\cap F.
\] 
We construct a strictly increasing sequence $(n_k)_k\subseteq\N$ and disjoint sets $C_k\subseteq B_{n_k}$ by induction. Let $n_0=0$ and $C_0=B_0$. Suppose that $n_0,\ldots,n_{k-1}$ and $C_0,\ldots,C_{k-1}$ have been constructed. Let $F=\bigcup_{i=1}^{k-1} C_{i}$. Pick $n_k> \max\{n_{k-1}, |F|\}$ such that if $B_{n_k}\cap F=L\cap F$, and let $C_k=B_{n_k}\setminus F$. Note that $n_k\geq k$ and the $C_k$'s are pairwise disjoint.

Since $L\cap F\subseteq B_{n_k}$ and $|L\cap F|< n_k$, we have $\| \Phi(\chi_{A_{n_k}})\chi_{L\cap F}a\|<\varepsilon/2$ for all $k\in\N$. Hence $\|\Phi(\chi_{A_{n_k}})\chi_{C_k}a\|>\varepsilon/2$ for all $k\in\N$. Also, if $D\subseteq C_k$ has at most $k$ elements, then $\|\Phi(\chi_{A_{n_k}})\chi_{D}a\|< \varepsilon/2-\delta$. Hence, the sequences $(A_{n_k})_k$ and $(C_k)_k$ contradict Lemma \ref{LemmaBla1}.
\end{proof}

The following proposition summarizes the previous two lemmas  in the setting of Assumption~\ref{Assumption}.

\begin{proposition}\label{prop:largesets}
In the setting of Assumption \ref{Assumption}: Suppose further that $X$ and $Y$ have property $A$. Then for all $\eps>0$ there exists  $k>0$ such that for all $A\subseteq X$ and $B\subseteq Y$ the following holds. 
\begin{enumerate}
\item\label{1.prop:largesets} If $\norm{\Phi(\chi_A)\chi_B\Phi(1)}>\varepsilon$ then there are $C\subseteq A$ and $D\subseteq B$ such that $\max\{|C|, |D|\}\leq k$ and $
\norm{\Phi(\chi_C)\chi_D\Phi(1)}>\e/8$.
\item \label{2.prop:largesets} If $\norm{\chi_A\Psi(\chi_B)}>\varepsilon$ then there are $C\subseteq A$ and $D\subseteq B$ such that $\max\{|C|,|D|\}\leq k$ and $ \norm{\chi_C\Psi(\chi_D)}>\eps/8$. 
\end{enumerate}
\end{proposition}

\begin{proof}Fix $\eps>0$.
Let $k_1\in\N$ be given by Lemma \ref{LemmaBla2} applied to $\Phi$, $a=\Phi(1)$, $\eps$ and $\eps/4$, and let $k_2\in\N$ be given by Lemma \ref{LemmaBla2} applied to $\Psi$, $a=\mathrm{Id}_{\cstu(X)}$, $\eps/4$ and $\eps/16$. We set $k=\max\{k_1,k_2\}$. 

\eqref{1.prop:largesets} Fix $A\subseteq X$ and $B\subseteq Y$ that satisfy $\norm{\Phi(\chi_A)\chi_B\Phi(1)}>\varepsilon$. Then there exists $D\subseteq B$ with $ |D|\leq k$ such that $\|\Phi(\chi_A)\chi_D\Phi(1)\|>\eps/4$. Since 
\[
\|\Phi(\chi_A)\chi_D\Phi(1)\|= \|\chi_{A}\Psi(\chi_D)\|,
\]
 our choice of $k_2$ implies that there exists $C\subseteq A$
with $|C|\leq k$ such that 
\[
\|\chi_{C}\Psi(\chi_D)\|>\e/4-\e/16>\e/8.
\]
Hence,   we have that
\[
\|\Phi(\chi_C)\chi_D\Phi(1)\|= \|\chi_{C}\Psi(\chi_D)\|>\e/8,
\]
and therefore \eqref{1.prop:largesets} follows. The proof of \eqref{2.prop:largesets} is analogous, so we omit it. 
\end{proof}

Given $A\subseteq X$ and $B\subseteq Y$, we define 
\[
X_{B,\delta}=\bigcup_{y\in B}X_{y,\delta}\ \text{ and }\ Y_{A,\delta}=\bigcup_{x\in A}Y_{x,\delta}.
\] 
In the metric case, the following corresponds to \cite[Lemma 7.6]{BragaFarahVignati2018}. 

\begin{lemma}\label{lemma:injective}
In the setting of Assumption \ref{Assumption}: Suppose further that both $X$ and $Y$ have property $A$. Then for every $\eps>0$ there is $\delta>0$ such that the following holds
\begin{enumerate}
\item \label{lemma:injective.1}
$\norm{\Phi(\chi_A)(1-\chi_{Y_{A,\delta}})}<\varepsilon$ 
for all $A\subseteq X$. 
\item \label{lemma:injective.2} $\norm{(1-\chi_{X_{B,\delta}})\Psi(\chi_B)}<\varepsilon$
for all $B\subseteq Y$. 
\item 	\label{lemma:injective.3} If $\e<1$ then $|A|\leq|Y_{A,\delta}|$ for all $A\subseteq X$. 
\end{enumerate}
Moreover, if $\Phi$ is an isomorphism, then $\delta$ can also be chosen such that $|B|\leq |X_{B,\delta}|$ for all $B\subseteq Y$.
\end{lemma}

\begin{proof}
Let $k>0$ be given by Proposition~\ref{prop:largesets}. 
 Then for all $A\subseteq X$ and $B\subseteq Y$ such that $\norm{\Phi(\chi_A)\chi_B\Phi(1)}\geq \varepsilon$ there are $C\subseteq A$ and $D\subseteq B$ of size $\leq k$ that satisfy $\norm{\Phi(\chi_C)\chi_D\Phi(1)}>\sqrt{\e}/8$. Let $\delta=\frac{\sqrt{\e}}{8k^2}$. 

\eqref{lemma:injective.1} Fix $A\subseteq X$ and $B\subseteq Y$, and let $C$ and $D$ be given as above. Enumerate $C$ and $D$  as $C=\{c_i\}_{i\leq k'}$, $D=\{d_i\}_{i\leq k''}$ with $k',k''\leq k$. Then 
\[
\delta k^2=\sqrt{\e}/8<\sum_{i\leq k'}\sum_{j\leq k''}\norm{\Phi(\chi_{c_i})\chi_{d_j}\Phi(1)}.
\]
Hence, there must exist $i\in \{1,\ldots,k'\}$ and $j\in \{1,\ldots,k''\}$ such that $\norm{\Phi(\chi_{c_i})\chi_{d_j}}>\delta$, i.e., $d_j\in Y_{c_i,\delta}$. Since $c_i\in A$, we have $d_j\in B\cap Y_{A,\delta}$. We have shown that $\norm{\Phi(\chi_A)\chi_B\Phi(1)}\geq \sqrt{\e}$ implies $B\cap Y_{A,\delta}\neq\emptyset$, so $\norm{\Phi(\chi_A)(1-\chi_{Y_{A,\delta}})\Phi(1)}<\sqrt{\e}$. Note that 
\begin{multline*}
\norm{\Phi(\chi_A)(1-\chi_{Y_{A,\delta}})}^2=\norm{\Phi(\chi_A)(1-\chi_{Y_{A,\delta}})\Phi(\chi_A)}\\
\leq \norm{\Phi(\chi_A)(1-\chi_{Y_{A,\delta}})\Phi(1)}<\sqrt{\e}, 
\end{multline*}
and 
\eqref{lemma:injective.1} follows. 

The proof of \eqref{lemma:injective.2} is analogous to the one of \eqref{lemma:injective.1}, so we omit it. 
Clause \eqref{lemma:injective.3} is an immediate consequence of \eqref{lemma:injective.1}, since the rank of $\chi_A$ is equal to~$|A|$ and the rank of $\chi_{Y_{A,\delta}}$ is equal to $|Y_{A,\delta}|$. 

For the moreover part, we just note that if $\Phi$ is an isomorphism, then we can apply \eqref{lemma:injective.1} and \eqref{lemma:injective.2} to the isomorphism $\Phi^{-1}$ (e.g., the argument at the end of \cite[Lemma 7.6]{BragaFarahVignati2018}).
\end{proof}

We are ready to prove Theorem~\ref{TheoremMainIsomor}.

\begin{proof}[Proof of Theorem \ref{TheoremMainIsomor}]
As explained  in the paragraph following Theorem \ref{TheoremMainIsomor}, we only need to prove the implication \eqref{2.Main} $\Rightarrow$ \eqref{1.Main}. Assume $X$ and $Y$ are uniformly locally finite, $Y$ has property A, and $\Phi\colon \cstu(X)\to \cstu(Y)$ is an isomorphism. By Theorem \ref{ThmNuclearPropA}, $\cstu(Y)$ is nuclear, and so is $\cstu(X)$, being isomorphic to it. In particular, again by Theorem~\ref{ThmNuclearPropA}, $X$ has property A. We are in the setting of Assumption~\ref{Assumption} where $\Psi=\Phi^{-1}$ and $\Phi(1)=1$. Hence, the hypotheses of Lemma~\ref{lemma:injective} are satisfied for both $\Phi$ and $\Phi^{-1}$. 

Pick $\eps<\frac{1}{2}$. Let $\delta>0$ be given by Lemma~\ref{lemma:injective}, such that $|A|\leq |Y_{A,\delta}|$ and $|B|\leq|X_{B,\delta}|$ for all $A\subseteq X$ and $B\subseteq Y$. Applying Hall's marriage theorem (\S\ref{S.Graphs}), we can construct injective maps $f\colon X\to Y$ and $g\colon Y\to X$ such that $f(x)\in Y_{x,\delta}$ and $g(y)\in X_{y,\delta}$ for all $x\in X$ and $y\in Y$. 
By Lemma~\ref{lemma:injective}, the hypotheses of Lemmas~\ref{lemma:entourages} and~\ref{lemma:coarse} are satisfied for both $\Phi$ and $\Phi^{-1}$, so both $f$ and $g$ are coarse and expanding. Moreover, we have that $\bigcup_{x\in X}(x,g\circ f(x))$ is an entourage since it is contained in $\bigcup_{x\in X}X_{f(x),\delta}\times X_{f(x),\delta}$, which is an entourage by Lemma~\ref{lemma:entourages}. Therefore $\mathrm{Id}_X$ and $g\circ f$ are close. Similarly, $\bigcup_{x\in X}(y,f\circ g(y))$ is an entourage, so $\mathrm{Id}_X$ and $f\circ g$ are close to each other. This show that $f$ and $g$ are injective coarse equivalences.
 
 By the standard proof of Cantor--Berstein's Theorem we can construct a bijection $h\colon X\to Y$ such that for all $x\in X$ we have that $h(x)=f(x)$ or $h(x)=g^{-1}(x)$. Such bijection is then close to $f$, and $h^{-1}$ is close to $g$, hence $h$ is a bijective coarse equivalence between $X$ and $Y$, and we are done.
\end{proof}

\begin{proof}[Proof of Theorem \ref{TheoremMainEmbed}]

The implication \eqref{ItemTheoremMainEmbed1} $\Rightarrow$ \eqref{ItemTheoremMainEmbed2}  is straightforward. Indeed, if $f\colon X\to Y$ is an injective coarse embedding, let $U\colon\ell_2(X)\to \ell_2(Y)$ be the isometry defined by $U\delta_x=\delta_{f(x)}$ and define $\Phi\colon\cstu(X)\to \cstu(Y)$ by $\Phi(a)=UaU^*$ for all $a\in \cstu(X)$; this $\Phi$ is the required embedding. Notice that this implication does not require any geometric assumption on the spaces of interest.

\eqref{ItemTheoremMainEmbed2} $\Rightarrow$ \eqref{ItemTheoremMainEmbed1} Let  $\Phi\colon \cstu(X)\to \cstu(Y)$ is an embedding onto a hereditary \cstar-subalgebra of $\cstu(Y)$ and assume that $Y$ has property A. As nuclearity passes to hereditary \cstar-subalgebras, Theorem \ref{ThmNuclearPropA} implies that $(X,\cE)$ also has property A, so we are in the setting of Assumption  \ref{Assumption}. 

Let us note that  $X$ injectively coarsely embeds into $Y$. Since both $(X,\cE)$ and $(Y,\cE')$  have property A, Lemma \ref{lemma:injective} applies, and we can pick $\delta>0$ such that $|A|\leq |Y_{A,\delta}|$ for all $A\subseteq X$. By Hall's Marriage Theorem, there is an injective $f\colon X\to Y$ such that $f(x)\in Y_{x,\delta}$ for all $x\in X$. Since Lemma~\ref{lemma:injective} applies, the hypotheses of Lemma~\ref{lemma:coarse} are satisfied for $\Phi$. Hence $f$ is coarse and expanding. This finished the proof.
\end{proof}

\subsection{The case of only compact ghost projections}
We finish this section obtaining rigidity under a geometric condition which is strictly weaker than property A -- unfortunately, we do not know if this geometric condition gives us strong rigidity, e.g., that $X$ and $Y$ are bijectively coarse equivalent. For that, we note that in order to apply Lemmas~\ref{lemma:entourages} and ~\ref{lemma:coarse}, we only need the thesis of Lemma~\ref{lemma:injective} for singletons. 

\begin{lemma}\label{lemma:injections2}
In the setting of Assumption \ref{Assumption}: Suppose further that $\cstu(X)$ and $\cstu(Y)$ have only compact ghost projections. Let $\eps>0$. Then there is $\delta>0$ such that for all $x\in X$ and $y\in Y$, 
\[
\norm{ (1-\chi_{X_{y,\delta}})\Psi(e_{yy})}, \norm{(1-\chi_{Y_{x,\delta}})\Phi(e_{xx})}<\eps.
\]
\end{lemma}
\begin{proof}
We prove the following   stronger statement: under the same  geometrical hypotheses on $\cstu(X)$ and $\cstu(Y)$, if $\Theta\colon\cstu(X)\to\cstu(Y)$ is a rank decreasing strongly continuous positive contraction such that $\Theta(a)=\Theta(1)\Theta(a)$ for all $a\in \cstu(X)$ and $\eps>0$ is given, then we can find $\delta>0$ such that for all $x\in X$ we have that
\[
\norm{(1-\chi_{Y_{x,\delta}})\Theta(e_{xx})}<\eps,
\]
where throughout this proof we use the notation
\[
Y_{x,\delta}=\Big\{y\in Y\mid \norm{ \Theta(e_{xx})e_{yy}\Theta(1)} >\delta\Big\}.
\]

Fix $\eps>0$ and suppose the thesis does not hold. For each $x\in X$, let $Z_x=\bigcup_\delta Y_{x,\delta}$. Note that $\chi_{Z_x}\Theta(e_{xx})=\Theta(e_{xx})$, as if $\norm{\Theta(e_{xx})e_{yy}}>0$ for some $y$, then 
\[
0<\norm{(\Theta(e_{xx})e_{yy})(\Theta(e_{xx})e_{yy})^*}=\norm{\Theta(e_{xx})e_{yy}\Theta(e_{xx})}\leq\norm{\Theta(e_{xx})e_{yy}\Theta(1)}
\]
and therefore $y\in Z_x$. Since each $\Theta(e_{xx})$ is compact, every $x\in X$ satisfies $\norm{\chi_{Y_{x,\delta}}\Theta(e_{xx})-\Theta(e_{xx})}\to 0$ as $\delta\to 0$. If the thesis does not hold, then there is a sequence $\{x_n\}_n\subseteq X$ such that $\|(1-\chi_{Y_{x_n,1/n}})\Theta(e_{x_nx_n})\|>\eps$ for all $n\in \N$. Since $\Theta$ is positive, then 
\[
\norm{(1-\chi_{Y_{x_n,1/n}})\Theta(e_{x_nx_n})^2(1-\chi_{Y_{x_n,1/n}})}>\eps^2.
\]
For each $n$, let $C_n\subseteq Y$ be a finite set such that 
\[
\max\{\norm{(1-\chi_{C_n})\Theta(e_{x_nx_n})}, \norm{\Theta(e_{x_nx_n})(1-\chi_{C_n})}\}<2^{-n}.
\]
By applying Lemma~\ref{L:AA} twice and passing to a subsequence we can assume that the $C_n$'s are disjoint, and that $\norm{\chi_{C_n}\Theta(\chi_B)}<2^{-n}$ whenever $B\subseteq \{x_m\}_{m\neq n}$. Let $Z=\{x_n\mid n\in\N\}$, $D_n=C_n\cap (Y\setminus Y_{x_n,1/n})$, $D=\bigcup_n D_n$ and $a_n=\chi_{D_n}\Theta(e_{x_nx_n})^2\chi_{D_n}$. Note that 
\[
\chi_D\Theta(\chi_Z)^2\chi_D-\sum_n a_n\in\mathcal K(\ell_2(Y)),
\]
so $\sum_n a_n$ belongs to $\cstu(Y)$. By exactly the same argument, we have that for every $M\subseteq\N$, the element 
\[
\chi_D\Theta(\chi_{\{x_n\}_{n\in M}})\chi_D-\sum_{n\in M} a_n\in\mathcal K(\ell_2(Y)),
\]
 hence $\sum_{n\in M}a_n$ belongs to $\cstu(Y)$. As $\norm{(1-\chi_{C_n})\Theta(e_{x_nx_n})}<2^{-n}$ for all $n$ we have that for all but finitely many $n$, $\norm{a_n}>\eps^2/2$. As $\Theta(e_{x_nx_n})^2$ is a  positive rank 1 operator, so is  $a_n$. More than that, as $D_n\cap Y_{x_n,1/n}=\emptyset$ and $a_n$ is the cut down of $\Theta(e_{x_nx_n})^2$ to the entries $(y,y')$ where $y,y'\in D_n$, we have that for all $\gamma>0$ there are only finitely many $n$'s such that there are $y,y'\in D_n$ for which $\norm{e_{yy}a_ne_{y'y'}}>\gamma$. As the supports of $a_n$ and $a_m$ are disjoint for each $n\neq m$, we have that $\sum_n a_n$ is a ghost. 
 
Let $\delta>\eps^2$ be a point of accumulation of the sequence $\norm{a_n}$; such a $\delta$ exists as $\norm{a_n}\in [\eps^2/2,1]$ for all $n$. Pick a subsequence $(a_{n_k})_k $ such that $|\norm{a_{n_k}}-\delta|<2^{-k}$ for all $k\in\N$. Then $\frac{1}{\delta}a_{n_k}$ is a rank $1$ positive element whose spectrum is $\{0,\gamma_k\}$, where $\gamma_k\to 1$ as $k\to \infty$. As all the $a_n$'s have disjoint support, then there is a compact element $d\in\mathcal K(\ell_2(Y))$ such that $\frac{1}{\delta}\sum_k a_{n_k}-d$ is a noncompact projection. More than that, such a $d$ can be found to have support equal to the union of the supports of $a_{n_k}$, and therefore can be found such that $d\in \mathcal K(\ell_2(Y))\cap\cstu(Y)$.  As $d\in\cK(\ell_2(Y))$ and $\frac{1}{\delta}\sum_k a_{n_k}$, then $ \frac{1}{\delta}\sum_k a_{n_k}-d$ is a noncompact ghost projection, and as $d\in\cstu(Y)$, then $\frac{1}{\delta}\sum_k a_{n_k}-d$ belongs to $\cstu(Y)$. This is a contradiction.
 
This completes the proof of the stronger claim. Since $\Phi$ and $\Psi$ satisfy the hypothesis of the claim, we are done.
\end{proof}

We are now able to prove our `weak rigidity' results.
\begin{proof}[Proof of Theorem~\ref{TheoremMainIsomor2}]
Let $X$ and $Y$ be u.l.f. coarse spaces such that all ghost projections in $\cstu(X)$ and $\cstu(Y)$ are compact. Suppose that $\Phi$ is an isomorphism. By Lemma~\ref{lemma:injections2} there is $\delta>0$ such that each $Y_{x,\delta}$ and each $X_{y,\delta}$ is nonempty. Let $f\colon X\to Y$ and $g\colon Y\to X$ be maps such that $f(x)\in Y_{x,\delta}$ and $g(y)\in X_{y,\delta}$ for all $x\in X$ and $y\in Y$. By Lemma~\ref{lemma:injections2}, the hypotheses of Lemma~\ref{lemma:coarse} are satisfied, therefore $f$ and $g$ are coarse and expanding. Moreover, as in the proof of Theorem \ref{TheoremMainIsomor},  Lemma~\ref{lemma:entourages} implies that $\mathrm{Id}_X$ is close to $g\circ f$, and $\mathrm{Id}_Y$ is close to $f\circ g$. Hence $X$ and $Y$ are coarsely equivalent.
\end{proof}

\begin{proof}[Proof of Corollary~\ref{cor:metr}]
Suppose that $\cstu(X)$ and $\cstu(Y)$ have only compact ghost projections, and they are isomorphic. By Theorem~\ref{TheoremMainIsomor2}, $(X,\cE)$ and $(Y,\cE')$ are coarsely equivalent. By Lemma~\ref{LemmaSize}, if $\cE$ is $\lambda$-generated, so is $\cE'$. This concludes the proof.
\end{proof}

We end this section with the proof of Theorem~\ref{TheoremMainIsomorMorita} and Theorem~\ref{thm:embednoghostproj}.
\begin{proof}[Proof of Theorem~\ref{TheoremMainIsomorMorita}]
Suppose that $(X,\mathcal E)$ and $(Y,\mathcal E')$ are uniformly locally finite coarse spaces and assume that all ghost projections in both $\cstu(X,\cE)$ and $\cstu(Y,\cE')$ are compact. We need to prove that if $\cstu(X,\cE)$ and $\cstu(Y,\cE')$ are stably isomorphic then 
$(X,\cE)$ and $(Y,\cE') $ are coarsely equivalent.
Let $\csts(X,\cE)$ denote the \emph{stable Roe algebra of $(X,\cE)$}, i.e., $\csts(X,\cE)=\cstu(X,\cE)\otimes \cK(H)$, where $H$ is an infinite dimensional Hilbert space. 
 Let $\Phi\colon\csts(X,\cE)\to \csts(Y,\cE')$ be an isomorphism. Let $q\in M_1(\C)\subseteq \cK(H)$ be a rank 1 projection and $1_X\in \cstu(X)$ be the identity. For any $\eps>0$, there exists $n\in\N$ large enough such that, letting $p$ be the unit of $\cstu(Y)\otimes M_n(\C)$, we have that \[\|\Phi(1_X\otimes q)-p\Phi(1_X\otimes q)p\|\leq \eps.\]
Choosing $\eps$ small enough, standard functional calculus allows us to pick a projection $p'\in \cstu(Y)\otimes M_n(\mathbb C)$ such that $\|\Phi(1_X\otimes q)-p'\|<1$. By a well-known result (e.g., \cite[Lemma 1.5.7]{Fa:STCstar}), there exists a unitary $u\in \csts(Y)$ such that $u\Phi(1_X\otimes q)u^*=p'$. Define $\Psi=\mathrm{Ad}(u)\circ \Phi$, so $\Psi$ is an isomorphism between $\csts(X)$ and $\csts(Y)$. Hence, $\Psi\restriction \cstu(X)\otimes M_1(\C)$ is an embedding of $\cstu(X) = \cstu(X)\otimes M_1(\C)$ onto a hereditary subalgebra of $\cstu(Y)\otimes M_n(\C)$.

Let $Y'=Y\times \{1,\ldots, n\}$. We define a coarse structure on $Y'$ by setting $E\subseteq Y'\times Y'$ to be an entourage if its projection on $Y\times Y$ belongs to $\cE'$, where $Y$ is identified with $Y\times \{1\}$. With this identification, we see  $(Y,\cE')$ as a coarse subspace of $Y'$, and, by abuse of notation, we still denote the coarse structure of $Y'$ by $\cE'$. Clearly, $(Y',\cE')$ was constructed such that there is a canonical isomorphism between $\cstu(Y)\otimes M_n(\C)$ and $ \cstu(Y')$.
With the identifications above, consider 
\[
\Psi\restriction \cstu(X)\otimes M_1(\C)\colon\cstu(X)\to \cstu(Y').
\]
Since all ghost projections in $\cstu(Y')$ are compact, the hypotheses of Lem\-ma~\ref{lemma:coarse} are satisfied, and therefore there is $\delta>0$ and a coarse embedding $f\colon X\to Y'$ such that $f(x)\in Y_{x,\delta}$ for all $x\in X$, where $Y_{x,\delta}=Y_{x,\delta,\Psi}$ is defined as in Lemma~\ref{L.XyYx}.

Let $1_{Y'}\in \cstu(Y')\cong \cstu(Y)\otimes M_n(\C)$ be the identity, i.e., $1_{Y'}=1_Y\otimes 1_n$, where $1_Y$ and $1_n$ are the identities of $\cstu(Y)$ and $M_n(\C)$, respectively. Similarly as above, there exists $k\in\N$ and a projection $p''\in \cstu(X)\otimes M_k(\C)$ such that $\|\Psi^{-1}(1_{Y'})-p''\|<\delta/12$. Hence, there exists a self-adjoint unitary $v\in \csts(X)$ such that $p''=v\Psi^{-1}(1_{Y'})v$ and 
\begin{equation}\label{Eq157}
\|\Psi^{-1}(e_{yy} )-v\Psi^{-1}(e_{yy} )v\|<\frac{\delta}{2}
\end{equation} for all $y\in Y'$ (e.g., by \cite[Lemma 1.5.7]{Fa:STCstar}). Let $\Theta=\mathrm{Ad}(v)\circ \Psi^{-1}$, so $\Theta$ is an isomorphism between $\csts(Y)$ and $\csts(X)$. Moreover, $\Theta\restriction \cstu(Y)\otimes M_n(\C)$ is an embedding of $\cstu(Y')$ onto a hereditary subalgebra of $\cstu(X)\otimes M_k(\C)\cong\cstu(X')$, where $X'=X\times \{1,\ldots, k\}$. Since all ghost projections in $\cstu(X')$ are compact, we can apply Lemma~\ref{lemma:coarse}, and find a positive $\gamma< \delta/2$ and a coarse embedding $g\colon Y'\to X'$ such that $g(y)\in X_{y,\gamma}$ for all $y\in Y'$, where $ X_{y,\gamma}=X_{y,\gamma,\Theta}$ is defined as in \eqref{EqXdelta} of Lemma~\ref{L.XyYx}.

Let $\pi\colon X'\to X$ be the projection, such that $\pi $ is a coarse equivalence. Let $h=\pi \circ g$, so $h$ is a coarse embedding of $Y'$ into $X$. 

\begin{claim}
The map $h\circ f$ is close to $\mathrm{Id}_{X}$.
\end{claim} 

\begin{proof}
By Lemma \ref{lemma:injections2} and Lemma \ref{lemma:entourages} applied to $\Theta\restriction \cstu(Y)\otimes M_n(\C)$, the set $\bigcup_{y\in Y'}X_{y,\gamma}\times X_{y,\gamma}$ is an entourage of $X'$. By the definition of the coarse structure of $X'$, this implies that 
$\bigcup_{y\in Y'}\pi(X_{y,\gamma})\times \pi(X_{y,\gamma})$
is an entourage of $X$. Hence, 
 we only need to notice that \[ (x,h\circ f(x))\in \bigcup_{x}\pi(X_{f(x),\gamma})\times \pi(X_{f(x),\gamma})\] for all $x\in X$. By the definition of $g$, $h\circ f(x)\in \pi(X_{f(x),\gamma})$ for all $x\in X$. Also, as $f(x)\in Y_{x,\delta}$ for all $x\in X$, we have that
\[
\delta\leq \|e_{f(x)f(x)}\Psi(e_{xx}\otimes q)\|\leq \|\Psi^{-1}(e_{f(x)f(x)})e_{xx}\otimes q\|.
\]
By \eqref{Eq157}, this implies that
\[\|\Theta(e_{f(x)f(x)})e_{xx}\otimes q\|\geq \frac{\delta}{2}\]
for all $x\in X$, and we can similarly obtain that $\|e_{xx}\otimes q\Theta(e_{f(x)f(x)})\|\geq \delta/2$ as well. Hence, $x\in \pi(X_{f(x),\gamma})$ for all $x\in X$, and we are done.
\end{proof} 

We can now conclude that $X$ and $Y$ are coarsely equivalent. The previous claim implies that $h\colon Y'\to X$ is cobounded. Hence, as $h$ is a coarse embedding, this shows that $h$ is a coarse equivalence. Therefore, as the projection $Y'\to Y$ is a coarse equivalence, the conclusion follows.
\end{proof}

\begin{proof}[Proof of Theorem~\ref{thm:embednoghostproj}]
Suppose  $(X,\cE)$ coarsely embeds into $(Y,\cE')$, let $f\colon X\to Y$ be such embedding. Expansiveness of $f$  implies that there exists $n\in\N$ so that $f^{-1}(y)$ has at most $n$ elements for all $y\in Y$. For each $y\in Y$, let $n_y=|f^{-1}(y)|$ and enumerate $f^{-1}(y)$, say $f^{-1}(y)=\{z^y_1,\ldots, z^y_{n_y}\}$. Let $Y'=Y\times \{1,\ldots,  n\}$ and endow $Y'$ with a coarse structure just as in the proof of  Theorem \ref{TheoremMainIsomorMorita}. Let $\alpha\colon X\to \{1,\ldots, n\}$ be the map such that $x=z^{f(x)}_{\alpha(x)}$ for all $x\in X$. Then the map 
\[x\in X\mapsto \Big(f(x), z^{f(x)}_{\alpha(x)}\Big)\in Y'\]  
is an injective coarse  embedding.  Hence, $\cstu(X)$ is isomorphic to a hereditary \cstar-subalgebra of $\cstu(Y')=\cstu(Y)\otimes M_n(\C)$. This gives us the implication \ref{Item1:thm:embednoghostproj} $\Rightarrow$ \ref{Item2:thm:embednoghostproj}.

 The implication \ref{Item2:thm:embednoghostproj} $\Rightarrow$ \ref{Item1:thm:embednoghostproj} follows analogously to the proof of Theorem \ref{TheoremMainIsomorMorita}, so we leave the details to the reader. 
\end{proof}

\section{Cartan subalgebras of uniform Roe algebras}\label{SectionAppli}
 Definition~\ref{Def.Cartan} below, isolated in \cite{renault2008cartan} building on the work of  \cite{kumjian1986c}, is the \cstar-algebraic variant of the standard notion from the theory of von Neumann algebras. If $B$ is a \cstar-subalgebra of $A$ then $E\colon A\to B$ is a \emph{conditional expectation} if it is completely positive and satisfies $\Upsilon(bab')=b \Upsilon(a) b'$ for all $a$ in $A$ and all $b$ and $b'$ in $B$ (see e.g.,~\cite[\S3.3]{Fa:STCstar}).

\begin{definition}\label{Def.Cartan}
Let $A$ be a \cstar-algebra and $B\subseteq A$ be a \cstar-subalgebra. We say that~$B$ is a \emph{Cartan subalgebra} of $A$ if
\begin{enumerate}
\item $B$ is a maximal abelian self-adjoint subalgebra (masa) of $A$,
\item $A$ is generated as a \cstar-algebra by the \emph{normalizer} of $B$ in $A$, i.e., 
\[N_{A}(B)=\{u\in A\mid uBu^*\cup u^*Bu\subseteq B\},\]
\item\label{3.cartan} $B$ contains an approximate unit for $A$, and 
\item there is a faithful conditional expectation $\Upsilon\colon A\to B$. 
\end{enumerate}
\end{definition}

If $A$ is unital then \eqref{3.cartan} is clearly redundant. 
In the context of Roe algebras, White and Willett introduced the following in \cite[Definition 4.20]{WhiteWillett2017}.

\begin{definition}\label{DefiRoeSubalg}
Let $A$ and $B$ be \cstar-algebras and assume that $B$ is a Cartan subalgebra of $A$. We say that $(A,B)$ is a \emph{Roe Cartan pair}, or that $B$ is a \emph{Roe Cartan subalgebra of $A$}, if
\begin{enumerate}
\item\label{Item.DefiRoeSubalg1} $B$ is isomorphic as a \cstar-algebra to $\ell_\infty(\N)$, 
\item\label{Item.DefiRoeSubalg2} $A$ contains the \cstar-algebra of compact operators on a separable infinite
dimensional Hilbert space as an essential ideal, and 
\item\label{Item.DefiRoeSubalg3} $B$ is \emph{co-separable in $A$}, i.e., there is a countable set $S\subseteq A$ such that $A=C^*(B,S)$.\footnote{Equivalently, one may require $S\subseteq N_A(B)$.}
\end{enumerate}
\end{definition}

In \cite[Remark~3.4]{WhiteWillett2017}, White and Willett asked the following question: if $X$ is a u.l.f. metric space and $B$ is a Cartan subalgebra of $A=\cstu(X)$ satisfying \eqref{Item.DefiRoeSubalg1} and \eqref{Item.DefiRoeSubalg2} of Definition \ref{DefiRoeSubalg},  is $B$ automatically co-separable? We give a positive answer in the case when $X$ has property $A$. 

\begin{corollary}\label{CorWhillettWhite}
Let $(X,d)$ be a u.l.f. metric space with property $A$. Suppose that $B\subseteq\cstu(X)$ is a Cartan subalgebra of $\cstu(X)$ which is isomorphic to $\ell_\infty(\N)$. Then $B$ is co-separable in $\cstu(X)$.
\end{corollary}
\begin{proof}
Let $(X,d)$ and $B$ be as in the hypothesis. As $B\cong\ell_\infty(\NN)$, for each $i\in\N$, let $p_i\in\N$ be the minimal projection corresponding to the image of $e_i\in \ell_\infty(\N)$ under this isomorphism. Let $\mathcal E_B$ be the coarse structure on $\NN$ generated by the family of subsets
\[
\Big\{\{(i,j)\in\N\times \N\mid \norm{p_iTp_j}>\eps\}\mid \eps>0, T\in N_{\cstu(X)}(B)\Big\}.
\]
By \cite[Proposition 4.15]{WhiteWillett2017}, $\cstu(X,\cE_d)$ and $\cstu(\NN,\mathcal E_B)$ are isomorphic. As~$(X,\cE_d)$ has property~A, by applying Theorem~\ref{ThmNuclearPropA} twice we conclude that $(\NN,\cE_B)$ has property A. The hypotheses of Corollary~\ref{cor:metr} are satisfied,  hence~$\mathcal E_B$ is countably generated. By \cite[Lemma 4.19]{WhiteWillett2017}, $B$ is co-separable in~$\cstu(X)$.
\end{proof}

\subsection{Nonmetrizable reducts} 
\label{S.nonmetrizable} 

We conclude by presenting an example of uniform Roe algebras $\cstu(X)\subseteq \cstu(Y)$ which share a common Cartan masa	 and so that $Y$ is metrizable but $X$ is not. Note that there are $2^{2^{\aleph_0}}$ subsets of the power set of $\bbN^2$, and it is not hard to show that  there are $2^{2^{\aleph_0}}$ coarse structures on $\bbN$ (see Proposition \ref{P.nonmetrizable} for a stronger statement). 
Since every metrizable coarse structure is countably generated, only $2^{\aleph_0}$ of these spaces are metrizable. 

Let $\cE_{|\cdot|}$ denote the metrizable coarse structure induced by the standard metric on $\bbN$. If $E\subseteq\N\times\N$, consider the set of `splitting' points of $E$, given by 
\[
S(E)=\{n\in\N\mid \text{ if } i<n\leq j\text{ then }(i,j)\notin E\text{ and } (j,i)\notin E\}
\]
If $\mathcal F$ is a filter on $\N$, consider 
\[
\mathcal E_{\mathcal F}=\{E\in\cE_{|\cdot|}\mid S(E)\in \mathcal F\}
\]
(see Figure \ref{FigXY} for an illustration of $S(E)$). 
Note that $\mathcal F=\{S(E)\mid E\in\mathcal E_{\mathcal F}\}$, so if $\mathcal F_1\neq\mathcal F_2$ are distinct filters, then $\mathcal E_{\mathcal F_1}\neq \mathcal E_{\mathcal F_2}$. The following proposition summarises the properties of $\mathcal E_{\mathcal F}$.

\begin{figure}[ht]
   \begin{tikzpicture}[scale=0.23]
    \coordinate (Origin)   at (0,0);
   \coordinate (XAxisMin) at (0,0);
    \coordinate (XAxisMax) at (17,0);
    \coordinate (YAxisMin) at (0,0);
    \coordinate (YAxisMax) at (0,17);
       \draw [<->,thick] (0,17) node (yaxis) [left] {$\N$}
        |- (17,0) node (xaxis) [below] {$\N$};

 \foreach \x in {0,1,2, 3,4,5,6,7,8,9,10,11,12,13,14,15,16}{
      \foreach \y in {0,1,2, 3,4,5,6,7,8,9,10,11,12,13,14,15,16}{
        \node[draw,circle,inner sep=0pt,fill] at (\x,\y) {};
       }}  
  
     \foreach \x in {0,1,2}{
      \foreach \y in {0,1,2}{
        \node[draw,circle,inner sep=1pt,fill] at (\x,\y) {};
       }} 
     \foreach \x in {5,6,7}{
      \foreach \y in {5,6,7}{
        \node[draw,circle,inner sep=1pt,fill] at (\x,\y) {};
        }}
     \foreach \x in {9,10,11}{
      \foreach \y in {9,10,11}{
        \node[draw,circle,inner sep=1pt,fill] at (\x,\y) {};
}}
   
       \foreach \x in {14,15,16}{
      \foreach \y in {14,15,16}{
        \node[draw,circle,inner sep=1pt,fill] at (\x,\y) {};
}}

     \foreach \x in {3}{
      \foreach \y in {3}{
        \node[draw,circle,inner sep=1pt,fill] at (\x,\y) {};
       }} 
     \foreach \x in {4}{
      \foreach \y in {4}{
        \node[draw,circle,inner sep=1pt,fill] at (\x,\y) {};
       }}              
     \foreach \x in {8}{
      \foreach \y in {8}{
        \node[draw,circle,inner sep=1pt,fill] at (\x,\y) {};
}}
     \foreach \x in {12}{
      \foreach \y in {12}{
        \node[draw,circle,inner sep=1pt,fill] at (\x,\y) {};
}}
     \foreach \x in {13}{
      \foreach \y in {13}{
        \node[draw,circle,inner sep=1pt,fill] at (\x,\y) {};
}}

  \end{tikzpicture}
  \begin{tikzpicture}[scale=0.23]
    \coordinate (Origin)   at (0,0);
   \coordinate (XAxisMin) at (0,0);
    \coordinate (XAxisMax) at (17,0);
    \coordinate (YAxisMin) at (0,0);
    \coordinate (YAxisMax) at (0,17);
 \draw [<->,thick] (0,17) node (yaxis) [left] {$\N$}
        |- (17,0) node (xaxis) [below] {$\N$};

 \foreach \x in {0,1,2, 3,4,5,6,7,8,9,10,11,12,13,14,15,16}{
      \foreach \y in {0,1,2, 3,4,5,6,7,8,9,10,11,12,13,14,15,16}{
        \node[draw,circle,inner sep=0pt,fill] at (\x,\y) {};
       }}  
 
  \foreach \x in {0}{
      \foreach \y in {0}{
        \node[draw,gray,circle,inner sep=1pt] at (\x,\y) {};
       }} 
     \foreach \x in {3}{
      \foreach \y in {3}{
        \node[draw,gray,circle,inner sep=1pt] at (\x,\y) {};
       }} 
     \foreach \x in {4}{
      \foreach \y in {4}{
        \node[draw,gray,circle,inner sep=1pt] at (\x,\y) {};
       }}        
    \foreach \x in {5}{
      \foreach \y in {5}{
        \node[draw,gray,circle,inner sep=1pt] at (\x,\y) {};
       }}

     \foreach \x in {8}{
      \foreach \y in {8}{
        \node[draw,gray,circle,inner sep=1pt] at (\x,\y) {};
}}
   
     \foreach \x in {9}{
      \foreach \y in {9}{
        \node[draw,gray,circle,inner sep=1pt] at (\x,\y) {};
}}
     \foreach \x in {12}{
      \foreach \y in {12}{
        \node[draw,gray,circle,inner sep=1pt] at (\x,\y) {};
}}
     \foreach \x in {13}{
      \foreach \y in {13}{
        \node[draw,gray,circle,inner sep=1pt] at (\x,\y) {};
}}
     \foreach \x in {14}{
      \foreach \y in {14}{
        \node[draw,gray,circle,inner sep=1pt] at (\x,\y) {};
}}

     \foreach \x in {0,3,4,5,8,9,12,13,14}{
      \foreach \y in {0}{
        \node[draw,circle,inner sep=1pt,fill] at (\x,\y) {};
}}

  \end{tikzpicture}\caption{Large black dots represnt an  entourage $E\subseteq \N\times\N$ in $\cE_{|\cdot|}$  (left) and $S(E)\subseteq \N\equiv\N\times \{0\}$ (right).}\label{FigXY}
\end{figure}

Recall that a filter is countably generated if there is a countable set $\{A_n\}_n\subseteq\mathcal F$ such that for all $B\in\mathcal F$ there is $n$ such that $A_n\subseteq B$. Among non countably generated filters are nonprincipal ultrafilters.
\begin{lemma}\label{lemma:uf}
Let $\mathcal F$ be a filter. Then
\begin{enumerate}
\item \label{1.filter} $\mathcal E_{\mathcal F}$ is a coarse structure on $\N$, in particular, is a reduct of $\mathcal E_{|\cdot|}$;
\item \label{2.filter} if $\mathcal F$ contains all cofinite sets, $\mathcal E_{\mathcal F}$ is connected;
\item \label{3.filter} If $\mathcal E_{\mathcal F}$ is metrizable, then $\mathcal F$ is countably generated.\footnote{Metrizability of $\mathcal E_\mathcal F$ and the fact that $\mathcal F$ is countably generated are in fact equivalent statements; as we will not need this, we leave the proof to the reader.}
\end{enumerate}
\end{lemma}
\begin{proof}
By the definition of $S(E)$, we have that $S(E)=S(E^{-1})$ for all $E\subseteq\N\times\N$, so $\mathcal E_{\mathcal F}$ is symmetric. Moreover $S(E)\cap S(F)\subseteq S(E\cup F)$ and, if $E\subseteq F$, then $S(F)\subseteq S(E)$ for all $E,F\subseteq\N\times\N$, therefore $\mathcal E_{\mathcal F}$ is closed by finite unions and subsets, as $\mathcal F$ is closed by finite intersections and supersets. It remains to show that $\mathcal E_{\mathcal F}$ is closed by finite composition. For this, let $E,F\subseteq\N\times\N$ and $n\in S(E)\cap S(F)$. If $i<n\leq j$, then there is no $k\geq n$ such that $(i,k)\in E$. Similarly, there is no $k<n$ such that $(k,j)\in F$, and therefore we cannot have that $(i,j)\in E\circ F$. This shows that $S(E\circ F)\supseteq S(E)\cap S(F)$, hence if $E,F\in\mathcal E_{\mathcal F}$, so does $E\circ F$. This concludes the proof of (1).

\eqref{2.filter} is immediate from the fact that if $E\subseteq\N\times\N$ is finite, then $S(E)$ is cofinite. Hence if $\mathcal F$ contains all cofinite sets, then all finite sets are in $\mathcal E_{\mathcal F}$.

For \eqref{3.filter}, we use the characterisation of metrizability provided by Lemma~\ref{L.infty-valued-metric}. Suppose that $\mathcal E_{\mathcal F}$ has a countably cofinal set $\{E_n\}_n$, and let $A_n=S(E_n)$. Let $B\in\mathcal F$, and $E_B=\{(i-1,i)\mid i\notin B\}$. As $S(E_B)=B$, then $E_B\in\mathcal E_{\mathcal F}$, hence there is $n$ such that $E_B\subseteq E_n$, meaning that $A_n\subseteq B$. This shows that $\{A_n\}_n$ generates $\mathcal F$, and therefore \eqref{3.filter} holds.
\end{proof}

\begin{proposition}\label{P.nonmetrizable} 
Let $\cE_{|\cdot|}$ denote the metrizable coarse structure induced by the standard metric on $\bbN$. 
\begin{enumerate}
\item \label{1.P.nonmetrizable} There exists a nonmetrizable connected reduct of $\cE_{|\cdot|}$. 
\item \label{2.P.nonmetrizable} There are $2^{2^{\aleph_0}}$ coarsely inequivalent reducts as in \eqref{1.P.nonmetrizable}. 
\end{enumerate}
\end{proposition}

Before proving Proposition~\ref{P.nonmetrizable}, we state an equivalent reformulation that shows that the conclusion of Theorem~\ref{TheoremMainEmbed} fails if the image is not assumed to be hereditary. 

\begin{corollary}[cf. Proposition 2.4 of \cite{BragaFarahVignati2019}] There exists a coarse, connected, u.l.f. space $Y$ with property A such that $\cstu(Y)$ 
	is isomorphic to a \cstar-subalgebra of $\cstu(\bbN)$ via an isomorphism that sends $\ell_\infty(Y)$ onto $\ell_\infty(\bbN)$ but~$Y$ does not coarsely embed into $\bbN$. There are $2^{2^{\aleph_0}}$ coarsely inequivalent coarse spaces $(Y,\cE)$ with this property. \label{C.nonmetrizable} 
\end{corollary}

\begin{proof} Let $Y$ be any of the the spaces constructed in Proposition~\ref{P.nonmetrizable}, and embed $\cstu(Y)$ into $\cstu(\bbN)$ via the identity map. The space $Y$ has property~A since it is a reduct of a property A space, and the other properties are immediate.
\end{proof}

\begin{proof}[Proof of Proposition~\ref{P.nonmetrizable}] 
\eqref{1.P.nonmetrizable}	Let $\cU$ be a nonprincipal ultrafilter on $\bbN$. Lem\-ma~\ref{lemma:uf} \eqref{1.filter} and \eqref{2.filter} together imply that $\cE_\cU$ is a connected coarse structure included in~$\cE$. Since $\cU$ is not countably generated, $\cE_\cU$ is not metrizable by Lem\-ma~\ref{lemma:uf}~\eqref{3.filter}.

\eqref{2.P.nonmetrizable} There are $2^{2^{\aleph_0}}$ ultrafilters on $\bbN$. This is a standard consequence of the existence of an independent family $\cA$ of subsets of $\bbN$ of cardinality~$2^{\aleph_0}$ (see e.g., \cite[Chapter 9]{Fa:STCstar}). We claim that if $\cU$ and $\cV$ are distinct ultrafilters, then $\cE_\cU\neq \cE_\cV$. Assume $\cU\neq \cV$. Then there exists $A\in \cU$ such that $\bbN\setminus A\in \cV$. The set
 $E_1=\{(i,j)\mid |i-j|\leq 1\mid \max(i,j)\in A\}$ belongs to $\cE_\cV$ and the set  $E_2=\{(i,j)\mid |i-j|\leq 1\mid \max(i,j)\notin A\}$ belongs to $\cE_\cU$. These two sets cover the entourage that generates the standard metric structure on~$\bbN$. Therefore~\eqref{1.P.nonmetrizable} implies  $E_1\notin \cE_\cU$ and $E_2\notin \cE_\cU$, hence $\cE_\cU\neq \cE_\cV$.      Some of the spaces $\cE_\cU$  are coarsely equivalent (like those corresponding to Rudin--Keisler equivalent ultrafilters---see e.g., \cite[\S 9.4]{Fa:STCstar}). However, there are only~$2^{\aleph_0}$ functions from $\bbN$ into $\bbN$, and therefore every coarse equivalence class of spaces on $\bbN$ contains at most $2^{\aleph_0}$ elements. Since $2^{2^{\aleph_0}}$ cannot be covered by fewer than $2^{2^{\aleph_0}}$ sets of cardinality $2^{\aleph_0}$ each, there are $2^{2^{\aleph_0}}$ coarse equivalence classes among the spaces of the form $\cE_{\cU}$. 
\end{proof}

\begin{acknowledgments}
This work started when BB visited AV at KU Leuven, and continued when AV visited BB and IF in Toronto. These visits were supported by AV's FWO grant and IF's NSERC grant. IF is partially supported by NSERC. AV is partially supported by the ANR Project AGRUME (ANR-17-CE40-0026). The authors are grateful to the funding bodies and the hosting institutions. 
\end{acknowledgments}


\begin{thebibliography}{10}

\bibitem{Black:Operator}
B.~Blackadar.
\newblock {\em Operator algebras}, volume 122 of {\em Encyclopaedia of
  Mathematical Sciences}.
\newblock Springer-Verlag, Berlin, 2006.
\newblock Theory of \cstar-algebras and von Neumann algebras, Operator Algebras
  and Non-commutative Geometry, III.

\bibitem{BragaChungLi2019}
B.~M. {Braga}, Y.~{Chung}, and K.~{Li}.
\newblock {Coarse Baum-Connes conjecture and rigidity for Roe algebras}.
\newblock arXiv:1907.10237.

\bibitem{BragaFarah2018}
B.~M. {Braga} and I.~{Farah}.
\newblock {On the rigidity of uniform Roe algebras over uniformly locally
  finite coarse spaces}.
\newblock arXiv:1805.04236.

\bibitem{BragaFarahVignati2019}
B.~M. {Braga}, I.~{Farah}, and A.~{Vignati}.
\newblock {Embeddings of uniform Roe algebras}.
\newblock arXiv:1904.07291, to appear in Comm. Math. Phys.

\bibitem{BragaFarahVignati2018}
B.~M. {Braga}, I.~{Farah}, and A.~{Vignati}.
\newblock {Uniform Roe coronas}.
\newblock arXiv:1810.07789.

\bibitem{BragaVignati2019}
B.~M. Braga and A.~Vignati.
\newblock {On the uniform Roe algebra as a Banach algebra and embeddings of
  $\ell_p$ uniform Roe algebras}.
\newblock arXiv:1906.11725.

\bibitem{BrodzkiNibloWright2007}
J.~Brodzki, G.~A. Niblo, and N.~J. Wright.
\newblock Property {A}, partial translation structures, and uniform embeddings
  in groups.
\newblock {\em J. Lond. Math. Soc. (2)}, 76(2):479--497, 2007.

\bibitem{BrownGreenRieffel1977}
L.~G. Brown, P.~Green, and M.~A. Rieffel.
\newblock Stable isomorphism and strong {M}orita equivalence of
  {$C\sp*$}-algebras.
\newblock {\em Pacific J. Math.}, 71(2):349--363, 1977.

\bibitem{BrownOzawa}
N.~P. Brown and N.~Ozawa.
\newblock {\em {\cstar}-algebras and finite-dimensional approximations},
  volume~88 of {\em Graduate Studies in Mathematics}.
\newblock American Mathematical Society, Providence, RI, 2008.

\bibitem{eklof2002almost}
P.C. Eklof and A.H. Mekler.
\newblock {\em Almost free modules: set-theoretic methods}, volume~65 of {\em
  North--Holland Mathematical Library}.
\newblock Elsevier, 2002.

\bibitem{Fa:STCstar}
I.~Farah.
\newblock {\em Combinatorial Set Theory and \cstar-algebras}.
\newblock Springer Monographs in Mathematics. Springer, 2019.

\bibitem{HigsonRoe1995}
N.~Higson and J.~Roe.
\newblock On the coarse {B}aum-{C}onnes conjecture.
\newblock In {\em Novikov conjectures, index theorems and rigidity, {V}ol.\ 2
  ({O}berwolfach, 1993)}, volume 227 of {\em London Math. Soc. Lecture Note
  Ser.}, pages 227--254. Cambridge Univ. Press, Cambridge, 1995.

\bibitem{Kubota2017}
Y.~Kubota.
\newblock Controlled topological phases and bulk-edge correspondence.
\newblock {\em Comm. Math. Phys.}, 349(2):493--525, 2017.

\bibitem{kumjian1986c}
A.~Kumjian.
\newblock On \cstar-diagonals.
\newblock {\em Canad. J. Math.}, 38(4):969--1008, 1986.

\bibitem{magidor1994when}
M.~Magidor and S.~Shelah.
\newblock When does almost free imply free? (for groups, transversals, etc.).
\newblock {\em J. Amer. Math. Soc.}, pages 769--830, 1994.

\bibitem{renault2008cartan}
J.~Renault.
\newblock Cartan subalgebras in \cstar-algebras.
\newblock {\em Irish Math. Soc. Bull.}, (61):29--63, 2008.

\bibitem{Roe1993}
J.~Roe.
\newblock Coarse cohomology and index theory on complete {R}iemannian
  manifolds.
\newblock {\em Mem. Amer. Math. Soc.}, 104(497):x+90, 1993.

\bibitem{Roe1996}
J.~Roe.
\newblock {\em Index Theory, Coarse Geometry, and Topology of Manifolds},
  volume~90 of {\em CBMS Conference Proceedings}.
\newblock American Mathematical Society, 1996.

\bibitem{RoeBook}
J.~Roe.
\newblock {\em Lectures on coarse geometry}, volume~31 of {\em University
  Lecture Series}.
\newblock American Mathematical Society, Providence, RI, 2003.

\bibitem{RoeWillett2014}
J.~Roe and R.~Willett.
\newblock Ghostbusting and property {A}.
\newblock {\em J. Funct. Anal.}, 266(3):1674--1684, 2014.

\bibitem{Sako2013}
H.~{Sako}.
\newblock {Property A for coarse spaces}.
\newblock arXiv:1303.7027.

\bibitem{Sako2014}
H.~Sako.
\newblock Property {A} and the operator norm localization property for discrete
  metric spaces.
\newblock {\em J. Reine Angew. Math.}, 690:207--216, 2014.

\bibitem{shelah1980on}
S.~Shelah.
\newblock On a problem of {K}urosh, {J}\'{o}nsson groups, and applications.
\newblock In {\em Word problems, {II} ({C}onf. on {D}ecision {P}roblems in
  {A}lgebra, {O}xford, 1976)}, volume~95 of {\em Stud. Logic Foundations
  Math.}, pages 373--394. North-Holland, Amsterdam-New York, 1980.

\bibitem{SkandalisTuYu2002}
G.~Skandalis, J.~L. Tu, and G.~Yu.
\newblock The coarse {B}aum-{C}onnes conjecture and groupoids.
\newblock {\em Topology}, 41(4):807--834, 2002.

\bibitem{SpakulaWillett2013AdvMath}
J.~{\v{S}}pakula and R.~Willett.
\newblock On rigidity of {R}oe algebras.
\newblock {\em Adv. Math.}, 249:289--310, 2013.

\bibitem{WhiteWillett2017}
S.~{White} and R.~Willett.
\newblock {Cartan subalgebras of uniform Roe algebras}.
\newblock arXiv:1808.04410.

\bibitem{willett2019higher}
R.~Willett and G.~Yu.
\newblock Higher index theory.
\newblock Book draft, 2019.

\bibitem{Yu2000}
G.~Yu.
\newblock The coarse {B}aum-{C}onnes conjecture for spaces which admit a
  uniform embedding into {H}ilbert space.
\newblock {\em Invent. Math.}, 139(1):201--240, 2000.

\end{thebibliography}
\end{document}